\documentclass[reqno, 11pt]{amsart}

\usepackage[a4paper,left=3cm,right=3cm,top=4cm,bottom=4cm]{geometry}
\usepackage{xspace}
\usepackage[utf8]{inputenc}
\usepackage{graphicx}
\usepackage{amssymb}
\usepackage{mathrsfs}
\usepackage[active]{srcltx}
\usepackage{url}
\usepackage[colorlinks]{hyperref}
\usepackage{amsmath,amstext,amsxtra,amsgen,amsbsy,amsopn,amscd,amsthm,amsfonts}
\usepackage{latexsym}
\usepackage{enumitem}
\usepackage{xcolor}
\usepackage{bm}

\newtheorem{theorem}{Theorem}

\newtheorem{proposition}{Proposition}
\newtheorem{lemma}[proposition]{Lemma}

\theoremstyle{definition}

\theoremstyle{remark}

\newtheorem{remark}[proposition]{Remark}

\numberwithin{equation}{section}
\numberwithin{proposition}{section}

\newcommand\R{{\ensuremath {\mathbb R} }}
\newcommand\C{{\ensuremath {\mathbb C} }}

\renewcommand{\Re}{\mathrm{Re}}
\renewcommand{\Im}{\mathrm{Im}}

\renewcommand\phi{\varphi}

\renewcommand\le{\leqslant}

\renewcommand\epsilon{\varepsilon}
\renewcommand\hat{\widehat}
\renewcommand\tilde{\widetilde}
\renewcommand\bar{\overline}
\newcommand{\gH}{\mathfrak{H}}

\newcommand{\gS}{\mathfrak{S}}

\newcommand{\cD}{\mathcal{D}}

\newcommand\ii{{\ensuremath {\infty}}}

\renewcommand\d[1]{{\ensuremath{\,\text{d}#1}}}

\newcommand{\cC}{\mathcal{C}}

\newcommand{\cF}{\mathcal{F}}

\DeclareMathOperator{\im}{Im}

\DeclareMathOperator{\re}{Re}

\DeclareMathOperator{\Tr}{Tr}
\DeclareMathOperator{\tr}{Tr}

\newcommand{\dps}{\displaystyle}

\title{The Dirac-Klein-Gordon system in the strong coupling limit}

\author{Jonas Lampart}
\address{Jonas Lampart, CNRS \& Laboratoire Interdisciplinaire Carnot de Bourgogne, UMR CNRS  6303, Universit\'e de Bourgogne--Franche-Comt\'e, 9, avenue Alain Savary, 21078 Dijon Cedex, France.}
\email{lampart@math.cnrs.fr}

\author{Lo\"ic Le Treust}
\address{Lo\"ic Le Treust, Institut de Math\'ematiques de Marseille, UMR CNRS 7373, Aix-Marseille Universit\'e, Rue Frédéric Joliot-Curie, 13453 MARSEILLE Cedex 13}
\email{loic.le-treust@univ-amu.fr}

\author{Simona Rota Nodari}
\address{Simona Rota Nodari, Laboratoire Jean Alexandre Dieudonn\'e, UMR CNRS 7351,
            Université Côte d'Azur, 28, avenue Valrose, 06108 Nice Cedex 2, France.}
\email{simona.rotanodari@univ-cotedazur.fr}

\author{Julien Sabin}
\address{Julien Sabin, Centre de mathématiques Laurent Schwartz, UMR CNRS 7640, École polytechnique, 91128 Palaiseau Cedex, France}
\email{julien.sabin@polytechnique.edu}

\date{\today}

\begin{document}

\begin{abstract}
We study the Dirac equation coupled to scalar and vector Klein-Gordon fields in the limit of strong coupling and large masses of the fields. We prove convergence of the solutions to those of a cubic non-linear Dirac equation, given that the initial spinors coincide. This shows that in this parameter regime, which is relevant to the relativistic mean-field theory of nuclei, the retarded interaction is well approximated by an instantaneous, local self-interaction. We generalize this result to a many-body Dirac-Fock equation on the space of Hilbert-Schmidt operators. 
\end{abstract}

\maketitle
\section{Introduction}

In this article we consider a coupled system of partial differential equations consisting of a Dirac equation and equations of Klein-Gordon type. We will be interested in the asymptotic behaviour of solutions in the regime of large coupling constants, and large masses for the Klein-Gordon fields.
This is motivated by the equations of the relativistic mean-field theory of nuclei~\cite{Ring-96}. In this context, the Dirac equation models the dynamics of a nucleon and the Klein-Gordon equations the nuclear forces in a mean-field approximation. The large coupling constants and masses reflect the fact that the nuclear forces are strong but of short range.   
We will show that in this regime the solutions behave like the solutions to a cubic non-linear Dirac equation, i.e. the interaction mediated by the fields becomes an instantaneous self-interaction of the Dirac spinor.

We will also consider a variant of the equations which describes  the interaction of many particles with the Klein-Gordon fields in the Dirac-Fock formalism. 
In addition to the modelling of systems with several nucleons,  this is relevant in the context of the Dirac equation as its physical interpretation naturally leads to a many-particle formalism.
In fact, the Dirac operator is not bounded from below, which is problematic from a physical point of view, as it suggests that an arbitrary amount of energy can be extracted from a particle it describes. This led Dirac to postulate the existence of the positron and later to the development of quantum electrodynamics (QED)~\cite{dirac1928quantum,dirac1928quantumII}. Dirac's reasoning was that all states of negative energy must be already occupied, making them inaccessible to other (fermionic) particles, and a positron would correspond to a missing particle in the ``Dirac sea'' of negative-energy states.

Mathematically speaking, Dirac's reasoning in terms of particles occupying the energies of the one-particle spectrum means that one should describe the state of the total system by a projector to the space of occupied states, or a more general operator called the density matrix. A similar formalism is recovered from QED in the mean-field approximation~\cite{chaix1989}.
In general, these operators will have infinite rank, which can lead to a multitude of problems with the corresponding equations. Solving these problems in the most general case would go beyond the scope of this article, but we make a first step in this direction by considering the natural generalization of the Dirac-Klein-Gordon equations to the class of Hilbert-Schmidt operators. In this context we derive an approximation result analogous to the one-particle case.

\subsection{The one-body problem}
Consider the following system of a Dirac equation coupled to a scalar field $S$ and a vector field $\omega$, both satisfying Klein-Gordon equations,
\begin{equation}\label{eqRMFTgeneral}
    \left\{\begin{aligned}
        &i\partial_t \Psi=\bm \alpha \cdot (-i\nabla - \bm \omega)\Psi +\beta (m+S)\Psi +V\Psi\\
        &(\partial^2_t-\Delta+m_\sigma^2)S=-g_\sigma^2\rho_s(\Psi)\\
        &(\partial^2_t-\Delta+m_\omega^2)\omega=g_\omega^2 J(\Psi)
    \end{aligned}
    \right..
\end{equation}
These are the natural equations for a relativistic quantum particle, described by the wave-function $\Psi:\R\times \R^3\to \C^4$, coupled to relativistic (massive) classical fields $S:\R\times \R^3\to \R$ and $\omega=(V,\bm \omega):\R\times \R^3\to \R\times \R^3$. 
Here, $(\beta,\bm \alpha)$ are the complex $4\times 4$ Dirac matrices defined by
\begin{equation}
    \beta=\begin{pmatrix}
        {\rm Id}_{\C^2} & 0\\
        0 & -{\rm Id}_{\C^2}
    \end{pmatrix},\  
    \bm \alpha=(\alpha_1,\alpha_2,\alpha_3) \text{ with } \alpha_k=\begin{pmatrix}
        0 & \sigma_k\\
        \sigma_k & 0
    \end{pmatrix},
\end{equation}
for $k=1,2,3$, and 
\begin{equation}
    \sigma_1=\begin{pmatrix}
        0 & 1\\
        1 & 0
    \end{pmatrix},\quad
    \sigma_2=\begin{pmatrix}
        0 & -i\\
        i & 0
    \end{pmatrix},\quad
    \sigma_3=\begin{pmatrix}
        1 & 0\\
        0 & -1
    \end{pmatrix},
\end{equation}
the Pauli matrices. In what follows, we will denote by $D:=-i \bm \alpha \cdot \nabla +\beta m$ the Dirac operator (see~\cite{thaller} for a thorough introduction). 
Furthermore, 
\begin{equation}\label{def.rhos}
  \rho_s(\Psi)=\langle\beta\Psi,\Psi\rangle_{\C^4}
\end{equation}
is the scalar density, while $J=(\rho_v,\bm J)$ is the space-time current, with
\begin{equation}\label{def.j}
  \rho_v(\Psi)=\langle\Psi,\Psi\rangle_{\C^4}, \,\bm J(\Psi)=(J_1,J_2,J_3) \text{ with } \ J_k=\langle\Psi,\alpha_k\Psi\rangle_{\C^4}.
\end{equation}

In physics, these equations arise, for example, in the relativistic mean-field model of nuclear physics~\cite{Ring-96}. There, $\Psi$ is the wave-function of a nucleon, $S$ is a scalar field associated to the $\sigma$ meson, and $\omega$ is a vector field associated to the $\omega$ meson.
In this context, both the masses $m_\sigma>0, m_\omega>0$ of the fields and the associated coupling constants $g_\sigma, g_\omega$ are large. For this reason, the equations for $S$ and $\omega$ are usually replaced by 
\begin{equation}\label{eq:instant}
 \begin{aligned}
  S&=-\gamma_\sigma \rho_s(\Psi), \\
  \omega&=\gamma_\omega J(\Psi),
 \end{aligned}
\end{equation}
with $\gamma_\sigma=\tfrac{g_\sigma^2}{m_\sigma^2}$ and  $\gamma_\omega=\frac{g_\omega^2}{m_\omega^2}$, 
in which the fields are determined instantaneously by their $\Psi$-dependent sources.
This gives rise to the non-linear Dirac equation for $\Psi$
\begin{equation}\label{eq.dnl1}
 i\partial_t\Psi 
= \bm \alpha\cdot\left(
  -i\nabla-\gamma_\omega {\bm J}(\Psi)
  \right)\Psi
  +\beta\left(m - \gamma_\sigma \rho_s(\Psi) \right)\Psi
  +\gamma_\omega\rho_v(\Psi)\Psi
  \,.
\end{equation}
Our main result is that equation~\eqref{eq.dnl1} provides a good approximation to the behavior of $\Psi$, solving the Dirac equation in the system~\eqref{eqRMFTgeneral}, in the simultaneous strong-coupling and large-mass limit $m_\sigma, m_\omega, g_\sigma, g_\omega \to \infty$ with fixed ratios $\gamma_\sigma, \gamma_\omega$.

\begin{theorem}\label{thm.main1}
Let $s>\tfrac52$ and $\Psi_\mathrm{in}\in H^s(\R^3, \C^4)$, $(S_\mathrm{in},\dot{S}_\mathrm{in})\in H^s(\R^3, \R)\times H^{s-1}(\R^3, \R)$, $(\omega_\mathrm{in},\dot{\omega}_\mathrm{in})\in H^s(\R^3, \R^4)\times H^{s-1}(\R^3, \R^4)$.
Let
\begin{equation*}
 \Psi_\mathrm{nl}\in\mathcal{C}((-T_{\mathrm{min}}^{\mathrm{nl}}, T_{\mathrm{max}}^{\mathrm{nl}}),H^s(\mathbb{R}^3, \C^4))
\end{equation*}
be the maximal solution to~\eqref{eq.dnl1} with initial condition $\Psi_\mathrm{nl}\vert_{t=0}=\Psi_\mathrm{in}$. 
Let $\gamma_\sigma, \gamma_\omega \geq 0$, $m_\sigma, m_\omega>0$ and let
\[
	(\Psi,S,\omega)\in \mathcal{C}((-T_{\mathrm{min}}, T_{\mathrm{max}}),H^s(\mathbb{R}^3, \C^4)\times H^s(\mathbb{R}^3,\R)\times H^s(\mathbb{R}^3, \R^4) )\,,
\] 
be the maximal solution to~\eqref{eqRMFTgeneral} with $g_\sigma=m_\sigma\sqrt{\gamma_\sigma}$, $g_\omega=m_\omega\sqrt{\gamma_\omega}$   and initial conditions
\begin{equation*}
  \Psi_{|t=0}=\Psi_\mathrm{in},\ (S,\partial_t S)_{|t=0}= (S_\mathrm{in},\dot{S}_\mathrm{in}), \ (\omega,\partial_t \omega)_{|t=0}= (\omega_\mathrm{in},\dot{\omega}_\mathrm{in}).
\end{equation*}
Then, for all fixed $\gamma_\sigma, \gamma_\omega \geq 0$, we have
\[
\liminf_{m_\sigma, m_\omega\to \infty}T_{\rm min/max}\geq T_{\rm min/max}^{\mathrm{nl}}\,.
\]
and, for all $0<T_1<T_{\mathrm{min}}^{\mathrm{nl}}$, $0<T_2<T_{\mathrm{max}}^{\mathrm{nl}}$, and all  $0\leq s'<s$,
\[
\lim_{m_\sigma, m_\omega\to \infty}\|\Psi-\Psi_{\mathrm{nl}}\|_{\mathcal{C}([-T_1, T_2],H^{s'}(\mathbb{R}^3, \C^4))}=0\,.
\]
\end{theorem}

Remarkably, this result is independent of the initial conditions for $S$ and $\omega$ (that is, the convergence holds without requiring that \eqref{eq:instant} holds at the initial time). This is highly desirable from the point of view of physics, since otherwise the approximation might only hold for special initial conditions, and it is not immediately clear why a physical system should be in such an initial state. The assumption that the initial conditions for $\Psi$ in equations~\eqref{eqRMFTgeneral} and~\eqref{eq.dnl1} are exactly the same can easily be relaxed to include data whose difference tends to zero as $m_\sigma, m_\omega \to \infty$. We could also consider more general coupling constants $g_\sigma, g_\omega$ such that $g_\sigma^2/m_\sigma^2\to \gamma_\sigma$ and $g_\omega^2/m_\omega^2\to\gamma_\omega$.
These generalizations are essentially trivial, so we will not pursue them, in favor of a simpler presentation.

There are sufficient conditions guaranteeing that $T_{\mathrm{min}}^{\mathrm{nl}}, T_{\mathrm{max}}^{\mathrm{nl}}=\infty$ holds. For $\gamma_\omega=0$ and initial data that are small in $H^1$, this was proved in~\cite{bejenaru2015}, and under a different smallness condition in~\cite{candy2018}. For initial data that are small in $H^s$, $s>1$, this was proved in \cite{machihara2003} (the authors only treat the case $\gamma_\omega=0$, but their proof can be straightforwardly extended to the case $\gamma_\omega>0$). Together with the above theorem, $T_{\mathrm{min}}^{\mathrm{nl}}, T_{\mathrm{max}}^{\mathrm{nl}}=\infty$ implies that solutions to the Dirac-Klein-Gordon system exist for an arbitrarily long time provided that the masses $m_\sigma$ and $m_\omega$ are large enough (without requiring that $(S_\mathrm{in},\omega_\mathrm{in},\dot{S}_\mathrm{in},\dot{\omega}_\mathrm{in})$ is small). This is particularly interesting because global existence results for the Dirac-Klein-Gordon system are only known for small initial data $(\Psi_\mathrm{in},S_\mathrm{in},\omega_\mathrm{in},\dot{S}_\mathrm{in},\dot{\omega}_\mathrm{in})$ \cite{bejenaru2017}.

Note that the maximal solutions for both the system ~\eqref{eqRMFTgeneral} and the equation~\eqref{eq.dnl1} exist and are unique, see for instance Proposition~\ref{prop:exist} below. One can lower the regularity for local well-posedness \cite{machihara2003,bejenaru2015,bejenaru2017}, and it is a challenging problem to extend our convergence result accordingly to smaller values of $s$. The proof of Theorem~\ref{thm.main1} is given in Section~\ref{sect:one-body}, and actually provides a quantitative rate of convergence (see Remark \ref{rk:rate}). 

The reduction of the system~\eqref{eqRMFTgeneral} to the non-linear Dirac equation~\eqref{eq.dnl1} via~\eqref{eq:instant} amounts to dividing the equations for $S, \omega$ by $m_\sigma^2, m_\omega^2$ and neglecting the terms
\begin{equation}
 \frac{1}{m_\sigma^2}(\partial^2_t-\Delta)S\,,\frac{1}{m_\omega^2}(\partial^2_t-\Delta)\omega.
\end{equation}
Since these terms involve derivatives of $S, \omega$ it is not clear that they are really small, as the solutions can, and indeed will, oscillate wildly for large $m_\sigma, m_\omega$.
However, these fast oscillations do allow the fields to quickly adapt to any changes in the source term, while their effect on $\Psi$ essentially ``averages out''. So, while in general~\eqref{eq:instant} does not  hold even approximately, $\Psi$ is still well described by~\eqref{eq.dnl1} in the limit. Similar singular limits have been studied in the literature~\cite{schochet1986, added1988, daub2016,griesemer2017,baumstark2020}, and we will comment in Section~\ref{sec:heuristics} below on the particularity of our situation compared to some of these works.

\subsection{The many-body problem} 

The analog of Equation~\eqref{eqRMFTgeneral} in mean-field (Dirac-Fock) theory is
\begin{equation}\label{eq:dkg-mb}
 \begin{cases}
  i\partial_t\Gamma = [D+{ \beta}S +  V - {\bm \alpha}\cdot {\bm \omega},\Gamma],\\
  (\partial^2_t-\Delta+m_\sigma^2)S=-g_\sigma^2\rho_s(\Gamma)\\
        (\partial^2_t-\Delta+m_\omega^2)\omega=g_\omega^2 J(\Gamma),
 \end{cases}
\end{equation}
where $[\cdot, \cdot ]$ denotes the commutator, $S, V, {\bm \omega}$ are as above and the densities corresponding to the Hilbert-Schmidt operator $\Gamma$ with integral kernel $\Gamma(x,y)\in \C^{4\times 4}$ are formally given by
\begin{align}
 \rho_s(\Gamma)(x)&= \Tr_{\C^4}(\beta \Gamma(x,x)), \label{eq:rho-s-def} \\
 \rho_v(\Gamma)(x)&= \Tr_{\C^4}(\Gamma(x,x)), \label{eq:rho-v-def} \\
 J_k(\Gamma)(x)&=\Tr_{\C^4}(\alpha_k\Gamma(x,x)). \label{eq:rho-J-def}
\end{align}
Note that these expressions reduce to the ones of~\eqref{def.rhos},~\eqref{def.j} if $\Gamma$ is the rank-one orthogonal projection on $\C\Psi$. The corresponding non-linear Dirac equation for $\Gamma$ takes the form
\begin{equation}\label{eq:dnl-mb}
 i\partial_t\Gamma
= \left[D - { \beta}\gamma_\sigma \rho_s(\Gamma) +  \gamma_\omega \rho_v(\Gamma) -  \gamma_\omega {\bm \alpha}\cdot {\bm J}(\Gamma) , \Gamma \right].
\end{equation}

Let $\mathfrak{S}^p$, $p\geq 1$ denote the Schatten classes~\cite{simontrace}, in particular $\mathfrak{S}^2$ denotes the class of Hilbert-Schmidt operators.
We take $\Gamma$ as an element of the following space
\begin{equation}\label{eq:H^s-op}
 \mathfrak{H}^s:=\{\Gamma \in \mathfrak{S}^2(\R^3, \C^4): (1-\Delta)^{s/2}\Gamma(1-\Delta)^{s/2} \in \mathfrak{S}^2(\R^3, \C^4)\},
\end{equation}
endowed with the norm
\begin{equation}\label{eq:norm-H^s-op}
\|\Gamma\|_{\gH^s} := \| (1-\Delta)^{s/2}\Gamma(1-\Delta)^{s/2} \|_{\gS^2}, 
\end{equation}
where $s\geq 0$ is chosen large enough for the densities $\rho_s, J$ to make sense. For details on the functional setting and on the precise meaning of Equations \eqref{eq:dkg-mb} and \eqref{eq:dnl-mb}, we refer to Section \ref{sect:many-body}.

We then have the following generalization of Theorem~\ref{thm.main1}:
\begin{theorem}\label{thm.main-mb}
Let $s>\tfrac52$ and $\Gamma_\mathrm{in}\in \mathfrak{H}^s$ be a non-negative operator, $(S_\mathrm{in},\dot{S}_\mathrm{in})\in H^s(\R^3, \R)\times H^{s-1}(\R^3, \R)$, $(\omega_\mathrm{in},\dot{\omega}_\mathrm{in})\in H^s(\R^3, \R^4)\times H^{s-1}(\R^3, \R^4)$.
Let
\begin{equation*}
 \Gamma_\mathrm{nl}\in\mathcal{C}((-T_{\mathrm{min}}^{\mathrm{nl}}, T_{\mathrm{max}}^{\mathrm{nl}}),\mathfrak{H}^s))
\end{equation*}
be the maximal solution to~\eqref{eq:dnl-mb} with initial condition $\Gamma_\mathrm{nl}\vert_{t=0}=\Gamma_\mathrm{in}$. 
Let $\gamma_\sigma, \gamma_\omega \geq 0$, $m_\sigma, m_\omega>0$ and let
\[
	(\Gamma,S,\omega)\in \mathcal{C}((-T_{\mathrm{min}}, T_{\mathrm{max}}),\mathfrak{H}^s\times H^s(\mathbb{R}^3,\R)\times H^s(\mathbb{R}^3, \R^4) )\,,
\] 
be the maximal solution to~\eqref{eq:dkg-mb} with $g_\sigma=m_\sigma\sqrt{\gamma_\sigma}$, $g_\omega=m_\omega\sqrt{\gamma_\omega}$   and initial conditions
\begin{equation*}
  \Gamma_{|t=0}=\Gamma_\mathrm{in},\ (S,\partial_t S)_{|t=0}= (S_\mathrm{in},\dot{S}_\mathrm{in}), \ (\omega,\partial_t \omega)_{|t=0}= (\omega_\mathrm{in},\dot{\omega}_\mathrm{in}).
\end{equation*}
Then, for all fixed $\gamma_\sigma, \gamma_\omega \geq 0$, we have
\[
\liminf_{m_\sigma, m_\omega\to \infty}T_{\rm min/max}\geq T_{\rm min/max}^{\mathrm{nl}}\,.
\]
and, for all $0<T_1<T_{\mathrm{min}}^{\mathrm{nl}}$, $0<T_2<T_{\mathrm{max}}^{\mathrm{nl}}$, and all  $0\leq s'<s$,
\[
\lim_{m_\sigma, m_\omega\to \infty}\|\Gamma-\Gamma_{\mathrm{nl}}\|_{\mathcal{C}([-T_1, T_2],\mathfrak{H}^{s'})}=0\,.
\]
\end{theorem}

We will prove this Theorem in Section~\ref{sect:many-body}. As in the one-body case, our proof provides a quantitative rate of convergence (see Remark \ref{rk:rate-mb}). In particular, our proof includes a blow-up criterion for Hartree-type equations (Lemma \ref{lem:blowup-crit-mb}) that we have not encountered in the literature before, which relies on a Kato-Ponce type inequality for density matrices (Lemma \ref{lem:K-P-mb}).

We remark that the case of $\Gamma$ with finite rank equal to $N$ corresponds to a coupled system of $N$ Dirac equations, i.e. $N$-particles in Hartree-Fock approximation. This case could also be treated by a straightforward generalization of Theorem~\ref{thm.main1}. However, all of the relevant estimates will then depend on $N$. Theorem~\ref{thm.main-mb} is a generalization giving uniform control in $N$ and even allowing for $\Gamma$ of infinite rank  as long as $\Gamma \in \mathfrak{S}^2$. The next step would be to consider perturbations of the Dirac sea, e.g. modelled by the negative spectral projection of $D$, by Hilbert-Schmidt operators and discuss the corresponding renormalized equations, as in Bogoliubov-Dirac-Fock theory~\cite{chaix1989, hainzl2005, hainzl2007,GraHaiLewSer-13}.   

\subsection{Heuristics}\label{sec:heuristics}

Let us discuss heuristically why a convergence result like Theorem~\ref{thm.main1} is expected to hold, and what difficulties may arise. 
It is instructive to consider the integral equations associated to~\eqref{eqRMFTgeneral}. With only a scalar field (i.e. taking $\gamma_\omega = 0$ and $(\omega_\mathrm{in},\dot\omega_\mathrm{in})=0$) we have
\begin{align}
S(t)&=\cos (t\sqrt{-\Delta+m_\sigma^2})S_\mathrm{in}+\frac{\sin (t\sqrt{-\Delta+m_\sigma^2})}{\sqrt{-\Delta+m_\sigma^2}}\dot{S}_\mathrm{in}\nonumber\\
&\qquad -g_\sigma^2\int_0^t\frac{\sin( (t-s)\sqrt{-\Delta+m_\sigma^2})}{\sqrt{-\Delta+m_\sigma^2}}\rho_s(\Psi(s))\d s, 
\end{align}
which, integrating by parts, can be rewritten as
\begin{align*}
  S(t)&=-\gamma_\sigma(1-\Delta/m_\sigma^2)^{-1}\rho_s(\Psi(t))+\cos (t\sqrt{-\Delta+m_\sigma^2})\big(S_\mathrm{in}+\gamma_\sigma(1-\Delta/m_\sigma^2)^{-1}\rho_s(\Psi_\mathrm{in})\big)\\
  &\qquad +\frac{\sin (t\sqrt{-\Delta+m_\sigma^2})}{\sqrt{-\Delta+m_\sigma^2}}\dot{S}_\mathrm{in}
  +\gamma_\sigma\int_0^t\frac{\cos ((t-s)\sqrt{-\Delta+m_\sigma^2})}{1-\Delta/m_\sigma^2}\partial_s\rho_s(\Psi(s))\d s.
\end{align*}
Assume for the moment that~\eqref{eq:instant} holds for the initial data (even though this is not assumed in Theorem~\ref{thm.main1}). Then, since $(1-\Delta/m_\sigma^2)^{-1}\to 1$ strongly as $m_\sigma\to \infty$, $S(t)$ would indeed be approximately given by~\eqref{eq:instant} for all $t$ if the final integral is small when $m_\sigma$ tends to $+\infty$. It is actually enough to show that the integral $\int_0^t\cos ((t-s)\sqrt{-\Delta+m_\sigma^2})\partial_s\rho_s(\Psi(s))\d s$ is small. One way to approach this is to integrate by parts once more, to obtain
\begin{equation}
 \begin{aligned}
\int_0^t\cos ((t-s)\sqrt{-\Delta+m_\sigma^2})\partial_s\rho_s(\Psi(s))\d s
&=\frac{\sin(t\sqrt{-\Delta+m_\sigma^2})}{\sqrt{-\Delta+m_\sigma^2}}\partial_t\rho_s(\Psi)\vert_{t=0} \\
&\qquad +\int_0^t\frac{\sin ((t-s)\sqrt{-\Delta+m_\sigma^2})}{\sqrt{-\Delta+m_\sigma^2}}\partial_s^2\rho_s(\Psi(s))\d s,
 \end{aligned}
\end{equation}
which is small as long as $\partial_t^2\rho_s(\Psi)$ can be appropriately bounded.
Using the equation for $\Psi$ (and the fact that $S$ is real), one calculates that
\begin{equation}
 \partial_t \rho_s(\Psi)= 2\re\langle \beta\Psi,\partial_t\Psi\rangle_{\C^4}
 =2\im\langle \beta\Psi,D\Psi\rangle_{\C^4},
\end{equation}
and
\begin{equation}\label{eq:Dt^2rho}
\partial_t^2 \rho_s(\Psi)=2\re\left[ \langle\beta D\Psi,D\Psi\rangle_{\C^4}-\langle\beta\Psi,D^2\Psi\rangle_{\C^4}-2\langle\Psi,S(i \bm \alpha\cdot\nabla \Psi\rangle_{\C^4})\right].
\end{equation}
Roughly, an $H^s$-bound on the error in $S$ requires a bound on $\partial_t^2 \rho_s(\Psi)$ in $H^{s-1}$ and thus a bound on $\Psi\in H^{s+1}$. However, controlling the effect of  $S$ in
\begin{equation}\label{eq:PsiDuhamel}
 \Psi(t)=e^{-itD}\Psi_\mathrm{in}+\int_0^te^{-i(t-s)D}{ \beta}S(s)\Psi(s)\d s 
\end{equation}
with respect to the $H^{s+1}$-norm would require an $H^{s+1}$-bound on $S$.
Hence, the chain of estimates does not close due to a loss of derivatives. In the works \cite{daub2016,griesemer2017,baumstark2020}, it turns out that this loss of derivatives does not occur due to an additional regularizing effect, and in these cases one can use such an argument.

This loss of derivatives can be dealt with by considering the differences, or reduced variables,
\begin{equation}\label{eq:var-red}
 \overline{S}:=S+\gamma_\sigma \rho_s(\Psi), \qquad \overline{\omega}:=\omega-\gamma_\omega J(\Psi)=:(\overline{V}, \overline{\bm \omega}).
\end{equation}
The equation for $\overline{S}$ is then, using the equation for $S$, 
\begin{equation}\label{eq:kg-reduced}
  (\partial_t^2-\Delta+m_\sigma^2)\bar{S}=(\partial_t^2-\Delta+m_\sigma^2)S+\gamma_\sigma(\partial_t^2-\Delta)\rho_s(\Psi)+g_\sigma^2 \rho_s(\Psi)=\gamma_\sigma(\partial_t^2-\Delta)\rho_s(\Psi).
\end{equation}
Now
\begin{equation}
 \Delta \rho_s(\Psi) = 2\re \big(\langle \beta \Psi, \Delta \Psi \rangle_{\C^4} +  \langle \beta \nabla\Psi, \nabla \Psi \rangle_{\R^3\times \C^4}\big),
\end{equation}
and combining this with~\eqref{eq:Dt^2rho} we see that the second derivatives of $\Psi$ cancel on the right hand side, as $D^2=-\Delta+m^2$.
Passing to the equation~\eqref{eq:kg-reduced} thus eliminates the (apparent) loss of one derivative. Additionally, it has the advantage that the right side is now of size one so that one may hope to show that if $\bar{S}$ is small at the initial time, it remains small for times of order one.

If we do not assume that $\overline{S}\vert_{t=0}$ is small, then $S(t)$ contains oscillating terms such as
\begin{equation}
 \cos (t\sqrt{-\Delta+m_\sigma^2})\big(S_\mathrm{in}+\gamma_\sigma(1-\Delta/m_\sigma^2)^{-1}\rho_s(\Psi_\mathrm{in})\big)
\end{equation}
that are not, in general, small. However, these can be treated by integrating by parts in the equation for $\Psi$~\eqref{eq:PsiDuhamel}, and we still obtain convergence of $\Psi$, without convergence of $S$.

\section{The one-body case}\label{sect:one-body}

\subsection{Well-posedness and uniform estimates}\label{sect:exist}

First recall the integral formulations on a time interval $I$ containing $0$ of the various equations we are working on. Let $s>\tfrac32$. For the Dirac-Klein-Gordon equation, $(\Psi,S,\omega)\in\cC(I,H^s)$ is a solution to \eqref{eqRMFTgeneral} with the same initial conditions as in Theorem \ref{thm.main1} if and only if for all $t\in I$,
 \begin{equation}
  \begin{cases}\label{eq:dkg-int}
\dps\Psi(t) = e^{-itD}\Psi_\mathrm{in} -i\int_0^te^{-i(t-t')D}[(-\bm \alpha \cdot  \bm \omega +\beta S+V)\Psi](t')
\d t'\,,
\\
\\
 \dps S(t)= 
\cos\left(t\sqrt{-\Delta + m^2_\sigma}\right) S_\mathrm{in}
+ \frac{\sin\left(t\sqrt{-\Delta + m^2_\sigma}\right)}{\sqrt{-\Delta + m^2_\sigma}}\dot{S}_\mathrm{in} 
\\\dps\qquad\qquad\qquad 
-g_\sigma^2\int_0^t\frac{\sin\left((t-t')\sqrt{-\Delta + m^2_\sigma}\right)}{\sqrt{-\Delta + m^2_\sigma}}\rho_s(\Psi(t'))\d t',\\
\\
 \dps\omega(t)= 
\cos\left(t\sqrt{-\Delta + m^2_\omega}\right)\omega_\mathrm{in}
+ \frac{\sin\left(t\sqrt{-\Delta + m^2_\omega}\right)}{\sqrt{-\Delta + m^2_\omega}}\dot{\omega}_\mathrm{in}
\\\dps\qquad\qquad\qquad 
+g_\omega^2\int_0^t\frac{\sin\left((t-t')\sqrt{-\Delta + m^2_\omega}\right)}{\sqrt{-\Delta + m^2_\omega}}J(\Psi(t')) \d t'.\\
\end{cases}
 \end{equation}
For the nonlinear Dirac equation, $\Psi\in\cC(I,H^s)$ is a solution to \eqref{eq.dnl1} with initial condition $\Psi\vert_{t=0}=\Psi_{\mathrm{in}}$ if and only if for all $t\in I$,
\begin{equation}\label{eq:dnl-int}
 \Psi(t) = e^{-itD}\Psi_\mathrm{in} -i\int_0^te^{-i(t-t')D}[(-\gamma_\omega\bm \alpha\cdot {\bm J}(\Psi)
  - \gamma_\sigma\beta\rho_s(\Psi)+\gamma_\omega\rho_v(\Psi))\Psi](t')\,\d t'.
\end{equation}

We begin by stating a simple result on existence of solutions to our equations.

\begin{proposition}\label{prop:exist}
 Let $s>\tfrac32$ and $(\Psi_\mathrm{in}, S_\mathrm{in}, \dot S_\mathrm{in}, \omega_\mathrm{in}, \dot \omega_\mathrm{in})$ be as in Theorem~\ref{thm.main1}.
 \begin{enumerate}[label=(\roman*)]
  \item\label{pt1.prop} There exist $T_{\mathrm{min}}^{\mathrm{nl}}, T_{\mathrm{max}}^{\mathrm{nl}}\in(0,+\infty]$ and a unique maximal solution 
\[
\Psi_{\mathrm{nl}}\in\mathcal{C}((-T_{\mathrm{min}}^{\mathrm{nl}}, T_{\mathrm{max}}^{\mathrm{nl}}),H^s(\mathbb{R}^3, \C^4))
\]
 to the Cauchy problem \eqref{eq.dnl1} with initial condition $\Psi_{\mathrm{nl}}\vert_{t=0}=\Psi_{\mathrm{in}}$. Furthermore, if $T_{\rm max/min}^{\mathrm{nl}}<+\infty$ then 
\[
\limsup_{t\to T_{\rm max/min}^{\mathrm{nl}}}\|\Psi_{\mathrm{nl}}(t)\|_{L^\infty}= +\infty\,.
\] 
\item\label{pt2.prop}  
For all $m_\sigma, m_\omega, g_\sigma, g_\omega$ there exist $T_{\mathrm{min}}, T_{\mathrm{max}}\in(0,+\infty]$ and a unique maximal solution 
\[
	(\Psi,S,\omega)\in \mathcal{C}((-T_{\mathrm{min}}, T_{\mathrm{max}}),H^s(\mathbb{R}^3, \C^4)\times H^s(\mathbb{R}^3,\R)\times H^s(\mathbb{R}^3, \R^4) )
\] 
to the Cauchy problem \eqref{eqRMFTgeneral} with the same initial conditions as in Theorem \ref{thm.main1}. Furthermore, if $T_{\rm max/min}<+\infty$ then 
\begin{equation}
	\limsup_{t\to T_{\rm max/min}}\|(\Psi,S,\omega)(t)\|_{L^\infty} 
	= +\infty\,. \label{eq:blowup-crit}
\end{equation}
 \end{enumerate}
\end{proposition}

This result can be proved by standard methods for semilinear equations~\cite{cazenave2003}, using a fixed point argument on the integral formulations \eqref{eq:dkg-int} and \eqref{eq:dnl-int} together with the fact that for $s>\tfrac32$, $H^s$ is an algebra. The $L^\ii$ blow-up criterion is obtained via the Kato-Ponce inequality \cite[Lemma X.4]{KatPon-88}: for $u,v\in H^{s}$ with $s>\tfrac32$, we have
\begin{equation}\label{eq:K-P}
\|uv\|_{H^s}\leq c_{\eqref{eq:K-P}}(\|u\|_{L^\infty}\|v\|_{H^s} + \|v\|_{L^\infty}\|u\|_{H^s})\,. 
\end{equation}
A proof of existence in $H^2$ can be found in~\cite{najman1992}. Existence results with weaker regularity assumptions are also available~\cite{escobedo1997,machihara2003, machihara2005, bejenaru2015, bejenaru2017, candy2018}.

To prove the convergence result (Theorem \ref{thm.main1}), we will need bounds that are uniform with respect to the asymptotic parameters $m_\sigma$ and $m_\omega$. This will require additional regularity and this explains why we assume that $s>\tfrac52$ in Theorem \ref{thm.main1} while the existence result Proposition \ref{prop:exist} only assumes $s>\tfrac32$.

As explained above, the key to this uniformity is to work with the reduced variables $\overline{S}=S+\gamma_\sigma(\Psi)$, $\overline{\omega}=\omega - \gamma_\omega J(\Psi)$. If $I$ is an open interval containing $0$ and if $(\Psi,S,\omega)\in\cC(I,H^s(\R^3,\C^4\times\R\times\R^4))$ is a solution to \eqref{eqRMFTgeneral}, then $(\Psi,\bar{S},\bar{\omega})$ also belongs to $\cC(I,H^s(\R^3,\C^4\times\R\times\R^4))$ and is a solution to the equation
\begin{align}\label{eqRMFTgeneralchanged} 
 \left\{\begin{aligned}
        &i\partial_t \Psi=D \Psi + W(\Psi,\overline{S}, \overline{\omega})  \Psi \\
        &(\partial_t^2-\Delta + m_\sigma^2)\overline{S}
	=  \gamma_\sigma(\partial_t^2-\Delta)\rho_{s}(\Psi)\\
        &(\partial_t^2-\Delta + m_\omega^2)\overline{\omega}
	=  -\gamma_\omega(\partial_t^2-\Delta)J(\Psi)
    \end{aligned}
    \right.
\end{align}
with
\begin{align}\label{eq:W-def}
W(\Psi,\overline{S}, \overline{\omega})&= \bm \alpha \cdot \left( - \overline{\bm \omega}-\gamma_\omega{\bm J}(\Psi)\right) +\beta \left(\overline{S}-\gamma_\sigma\rho_s(\Psi)\right) +\left(\overline{V}+ \gamma_\omega\rho_v(\Psi)\right),
\end{align}
and the initial conditions
\begin{equation}\label{eq:init-red}
 \begin{aligned}
&\overline{S}_\mathrm{in}=S_\mathrm{in} + \gamma_\sigma\rho_{s}(\Psi_\mathrm{in})\,,
\qquad\dot{\overline{S}}_\mathrm{in}=\dot{S}_\mathrm{in} + \gamma_\sigma\dot\rho_{s}\,,
\\
&\overline{\omega}_\mathrm{in}=\omega_\mathrm{in}-  \gamma_\omega J(\Psi_\mathrm{in})\,,
\qquad \dot{\overline{\omega}}_\mathrm{in}=\dot{\omega}_\mathrm{in}-  \gamma_\omega \dot{J}\,,
 \end{aligned}
\end{equation}
where 
\begin{equation}
 \begin{aligned}
 \dot\rho_{s}
&=2\Re\langle \beta i\Psi_\mathrm{in},\bm \alpha \cdot (-i\nabla - \bm \omega_\mathrm{in})\Psi_\mathrm{in}\rangle_{\C^4} \\
\dot{J}&=(\dot \rho_v, \dot {\bm J})=(\dot \rho_v, \dot {J_1},\dot {J}_2, \dot {J_3}) \\
\dot \rho_v &=- 2\re \langle \Psi_\mathrm{in}, {\bm \alpha}\cdot \nabla \Psi_\mathrm{in}\rangle_{\C^4} \\
 \dot {J}_k&= 2\Re\langle\alpha_k i\Psi_\mathrm{in},\bm \alpha \cdot (-i\nabla - \bm \omega_\mathrm{in})\Psi_\mathrm{in} +\beta (m+S_\mathrm{in})\Psi_\mathrm{in}\rangle_{\C^4}, \qquad k=1,2,3,
 \end{aligned}
\end{equation}
which, for a solution $\Psi$, corresponds to the derivatives of $\rho_s(\Psi)$ and $J(\Psi)$ at $t=0$.

\begin{lemma}\label{lem.red-eq}
Let $s>\tfrac52$ and $\gamma_\sigma, \gamma_\omega\geq 0$.
\begin{enumerate}[label=(\roman*)]
 \item There are functions $P=(P_\sigma,P_\omega)$, and $Q=(Q_\sigma, Q_\omega)$ (which are independent of $m_\sigma$, $m_\omega$) such that, for all $m_\sigma, m_\omega>0$,
$(\Psi,\overline{S},\overline{\omega})$ is a solution to~\eqref{eqRMFTgeneralchanged} if and only if it solves the equation
\begin{align}\label{eqRMFT-PQ}
 \left\{\begin{aligned}
        &i\partial_t \Psi=D \Psi + W(\Psi, \overline{S}, \overline{\omega}) \Psi \\
        &(\partial_t^2-\Delta + m_\sigma^2)\overline{S}
	=  P_\sigma\circ (\Psi, \overline{S}, \overline{\omega}, \nabla\Psi, \nabla \overline{S}, \nabla \overline{\omega}) + \partial_t \big(Q_\sigma\circ (\Psi, \overline{S}, \overline{\omega})\big)\\
        &(\partial_t^2-\Delta + m_\omega^2)\overline{\omega}
	=  P_\omega\circ (\Psi, \overline{S}, \overline{\omega}, \nabla\Psi, \nabla \overline{S}, \nabla \overline{\omega}) + \partial_t \big(Q_\omega\circ(\Psi, \overline{S}, \overline{\omega})\big)\,.
    \end{aligned}
    \right.
\end{align}
%
%
\item  There exists a constant $c_{\eqref{eq:Q-bound}}>0$ such that for all $(\Psi,\overline{S}, \overline{\omega})\in H^s$ we have
\begin{equation}\label{eq:Q-bound}
\| Q\circ(\Psi,\overline{S}, \overline{\omega})\|_{H^s} \leq c_{\eqref{eq:Q-bound}}  \left(\|\Psi\|^4_{H^{s}} + \|\Psi\|^2_{H^{s}} \|(\overline{S}, \overline{\omega})\|_{H^s}\right)\,.
\end{equation}
\item For all $M>0$ there exists a constant $c_{\eqref{eq:PQ-bound}}(M)>0$ so that for all $(\Psi,\overline{S}, \overline{\omega})\in H^s$ with
\begin{equation*}
 \|(\Psi,\overline{S}, \overline{\omega})\|_{W^{1,\infty}} \leq M\,,
\end{equation*}
the inequality
\begin{equation}\label{eq:PQ-bound}
\begin{aligned}
 \| Q \circ(\Psi,\overline{S}, \overline{\omega})\|_{H^s} +\| P\circ(\Psi, \overline{S}, \overline{\omega}, \nabla\Psi, \nabla \overline{S}, \nabla \overline{\omega})\|_{H^{s-1}} 
 \leq c_{\eqref{eq:PQ-bound}}(M) \|(\Psi,\overline{S}, \overline{\omega})\|_{H^{s}}
\end{aligned}
\end{equation}
holds.
\end{enumerate}
\end{lemma}

\begin{proof}
For $\Psi\in\mathbb{C}^4$, we denote by $F(\Psi, \Psi)$, one of the sesquilinear forms defined by
\[
\gamma_\sigma\langle \Psi, \beta\Psi\rangle_{\C^4}\,,\quad -\gamma_\omega|\Psi|^2\,,\quad -\gamma_\omega\langle\Psi, \alpha_1\Psi\rangle_{\C^4}\,,\quad -\gamma_\omega\langle\Psi, \alpha_2\Psi\rangle_{\C^4}\,,\quad -\gamma_\omega\langle\Psi, \alpha_3\Psi\rangle_{\C^4}\,.\] 
Assume now that $\Psi:I \times \R^3 \to \C^4 $ is a solution of 
\begin{equation}\label{eqPsiW}
 i\partial_t \Psi=D \Psi + W(\Psi,\overline{\omega},\overline{S}) \Psi
\end{equation}
with arbitrary functions $(\overline{S}, \overline{\omega})\in \mathcal{C}(I, H^{s}(\R^3,\R^5))$
. 
We then have
\[\begin{split}
\partial_t F(\Psi,\Psi) 
&= 2\Re F(\Psi,\partial_t\Psi)
= 2 \Im F(\Psi, (D + W) \Psi) \,.
\end{split}\]
We set 
\begin{equation}\label{eq:Q def}
 Q_\bullet(\Psi, \overline{S}, \overline{\omega}):=2\Im F(\Psi, W(\Psi,\overline{\omega},\overline{S})\Psi),
\end{equation}
where $\bullet$ corresponds to the choice of $F$. By the definition of $W$, $Q$ is a sum of quartic terms in $\Psi$, and terms that are linear in $\overline{S}, \overline{\omega}$ and quadratic in $\Psi$.
This already implies the bound~\eqref{eq:Q-bound} on $Q$, since $H^s$, $s>\tfrac32$, is a normed algebra.

We then find
\begin{align*}
 \partial_t^2 F(\Psi,\Psi) & =  2\partial_t  \Im F(\Psi, D  \Psi) + \partial_t Q_\bullet \\
 &= - 2 \Re F(\Psi,  D (D+W) \Psi) + 2 \Re F((D+W)\Psi, D  \Psi)
 + \partial_t Q_\bullet.
\end{align*}
With
\begin{equation*}
 \Delta F = 2\Re F(\Psi, \Delta \Psi) + \sum_{k=1}^3 2 \Re F(\partial_k \Psi, \partial_k \Psi)
\end{equation*}
and using that
$
D^2 = -\Delta+m^2
$
we get 
\begin{align*}
(\partial_t^2&\,-\Delta)F(\Psi,\Psi)
=\\
& \underbrace{2\Re F((D+W)\Psi,D\Psi)  -2\Re F(\Psi,(m^2+ D W)\Psi)
-\sum_{k=1}^32\Re F(\partial_k\Psi,\partial_k\Psi)}_{=:P_\bullet} + \partial_t Q_\bullet\,.
\end{align*}
 This defines $P_\sigma, P_\omega$ (with the respective choices of $F$), which, by definition of $W$, is a polynomial  in $\Psi, \overline{S}, \overline{\omega}, \nabla\Psi, \nabla \overline{S}, \nabla \overline{\omega}$ and their complex conjugates.
 
 In the other direction, assume that $(\Psi,\overline{S}, \overline{\omega})$ solves~\eqref{eqRMFT-PQ}. Then $\Psi$ is a solution of~\eqref{eqPsiW} and we can retrace our steps in the construction of $P,Q$, showing that $(\Psi,\overline{S}, \overline{\omega})$ is a solution to~\eqref{eqRMFTgeneralchanged}.
 
The bound~\eqref{eq:PQ-bound} on $P,Q$ follows from the fact that they are polynomials in their arguments (and the corresponding complex conjugates) and from the Kato-Ponce inequality \eqref{eq:K-P}.

\end{proof}

\begin{remark}
 In \eqref{eqRMFT-PQ}, we did not absorb the term $\partial_t Q$ into $P$ because it is not possible to control properly the time derivatives uniformly in the masses (indeed, the functions $\partial_t \bar{S}$ and $\partial_t \bar{\omega}$ oscillate quickly). The term $\partial_t Q$ will be dealt with using integration by parts. Furthermore, one sees from the previous lemma where the condition $s>\tfrac52$ comes from: since the nonlinear terms in $P$ involve first order derivatives of $(\Psi,\bar{S},\bar{\omega})$, estimating these terms in $H^{s-1}$ by the Kato-Ponce inequality \eqref{eq:K-P} requires that $s-1>\tfrac32$ and thus $s>\tfrac52$. 
\end{remark}

Let $(\Psi_\mathrm{in}, \overline{S}_\mathrm{in}, \dot{\overline{S}}_\mathrm{in}, \overline{\omega}_\mathrm{in}, \dot{\overline{\omega}}_\mathrm{in})$ be determined by the initial conditions as in~\eqref{eq:init-red}. To obtain bounds on $(\Psi,\bar{S},\bar{\omega})$ that are uniform in $m_\omega, m_\sigma$, we work on the integral formulation of equation~\eqref{eqRMFT-PQ}. For example, the equation for $\overline{S}$ is
\begin{align*}
 \overline{S}(t)&=\cos(t\sqrt{-\Delta + m_\sigma^2})\overline{S}_\mathrm{in}
+
\frac{\sin(t\sqrt{-\Delta + m_\sigma^2})}{\sqrt{-\Delta + m_\sigma^2}}\dot{\overline{S}}_\mathrm{in} \\
& \qquad+
\int_0^t
\frac{\sin((t-t')\sqrt{-\Delta + m_\sigma^2})}{\sqrt{-\Delta + m_\sigma^2}}(P_\sigma(t')+\partial_t Q_\sigma(t'))\d t',
\end{align*}
and an integration by parts in the term with $Q$ yields
\begin{align*}
& \int_0^t \frac{\sin((t-t')\sqrt{-\Delta + m_\sigma^2})}{\sqrt{-\Delta + m_\sigma^2}}\partial_t Q_\sigma(t')\d t' \\
&= - \frac{\sin(t\sqrt{-\Delta + m_\sigma^2})) }{\sqrt{-\Delta + m_\sigma^2}} Q_\sigma(\Psi_\mathrm{in}, \overline{\omega}_\mathrm{in}, \overline{S}_\mathrm{in})
+\int_0^t \cos((t-t')\sqrt{-\Delta + m_\sigma^2})Q_\sigma(\Psi, \overline{S}, \overline{\omega})(t')\d t'  \,.
\end{align*}

As we will see, the term with $P$ will benefit from the regularizing effect of the operator $(-\Delta+m_\sigma^2)^{-1/2}$, while  $Q$ does not involve derivatives of $(\Psi, \overline{S}, \overline{\omega})$. For $T_1,T_2>0$, the full set of integral equations for a solution $\chi:=(\Psi, \overline{S}, \overline{\omega})\in\cC([-T_1,T_2],H^s)$ to \eqref{eqRMFT-PQ} with initial conditions $(\Psi_\mathrm{in}, \overline{S}_\mathrm{in}, \dot{\overline{S}}_\mathrm{in}, \overline{\omega}_\mathrm{in}, \dot{\overline{\omega}}_\mathrm{in})$ is: for all $t\in[-T_1,T_2]$,

 \begin{equation}
  \begin{cases}\label{eq:Phidef}
&\dps\Psi(t) = e^{-itD}\Psi_\mathrm{in} -i\int_0^te^{-i(t-t')D}[W(\chi)\Psi](t')
\d t'\,,
\\
\\
&\dps W(\chi) =  
	\bm \alpha \cdot \left( - \overline{\bm \omega}-\gamma_\omega{\bm J}(\Psi)\right) +\beta \left(\overline{S}-\gamma_\sigma\rho_s(\Psi)\right) +\left(\overline{V}+ \gamma_\omega \rho_v(\Psi)
	\right)\,,
 \end{cases}
 \end{equation}
 and
 
 \begin{equation}
     \begin{cases}\tag{\ref{eq:Phidef}--continued}
& \dps\bar{S}(t)= 
\cos\left(t\sqrt{-\Delta + m^2_\sigma}\right)\overline S_\mathrm{in}
+ \frac{\sin\left(t\sqrt{-\Delta + m^2_\sigma}\right)}{\sqrt{-\Delta + m^2_\sigma}}\left( \dot{\overline{S}}_\mathrm{in} 
-Q_\sigma(\Psi_\mathrm{in}, \overline{\omega}_\mathrm{in}, \overline{S}_\mathrm{in})\right)
\\&\dps\qquad\qquad\qquad 
+\int_0^t\frac{\sin\left((t-t')\sqrt{-\Delta + m^2_\sigma}\right)}{\sqrt{-\Delta + m^2_\sigma}}P_\sigma(\chi, \nabla\chi)(t')\d t'\\
&\dps\qquad\qquad\qquad 
+\int_0^t\cos\left((t-t')\sqrt{-\Delta + m^2_\omega}\right) Q_\sigma(\chi)(t')\d t'
\,, \\
\\
& \dps\bar{\omega}(t)= 
\cos\left(t\sqrt{-\Delta + m^2_\omega}\right)\overline{\omega}_\mathrm{in}
+ \frac{\sin\left(t\sqrt{-\Delta + m^2_\omega}\right)}{\sqrt{-\Delta + m^2_\omega}}\left(\dot{\overline{\omega}}_\mathrm{in} - Q_\omega(\Psi_\mathrm{in}, \overline{\omega}_\mathrm{in}, \overline{S}_\mathrm{in})\right)
\\&\dps\qquad\qquad\qquad 
+\int_0^t\frac{\sin\left((t-t')\sqrt{-\Delta + m^2_\omega}\right)}{\sqrt{-\Delta + m^2_\omega}}P_\omega(\chi, \nabla\chi)(t')\d t'\\
&\dps\qquad\qquad\qquad 
+\int_0^t\cos\left((t-t')\sqrt{-\Delta + m^2_\omega}\right) Q_\omega(\chi)(t')\d t'
\,.
\end{cases}
 \end{equation}
 
 \begin{lemma}\label{lem:unif-exist}
Let $s>\tfrac52$ and $\gamma_\sigma, \gamma_\omega\geq 0$.
 For $(\Psi_\mathrm{in}, S_\mathrm{in}, \dot S_\mathrm{in}, \ \omega_\mathrm{in}, \dot \omega_\mathrm{in})$ as in Theorem~\ref{thm.main1} set 
 \begin{align*}
  R_0&:= \|\Psi_\mathrm{in}\|_{H^s}
+\|(\overline{\omega}_\mathrm{in}, \overline{S}_\mathrm{in} )\|_{H^s} 
+ \|(\dot{\overline{\omega}}_\mathrm{in}, \dot{\overline{S}}_\mathrm{in})\|_{H^{s-1}} 
+\|Q_\omega(\Psi_\mathrm{in}, \overline{\omega}_\mathrm{in}, \overline{S}_\mathrm{in} )\|_{H^{s-1}}\\ 
&\qquad+ \|Q_\sigma(\Psi_\mathrm{in}, \overline{\omega}_\mathrm{in}, \overline{S}_\mathrm{in} )\|_{H^{s-1}}\,,
 \end{align*}
 with the reduced variables given by~\eqref{eq:init-red}. For all $T_1,T_2>0$ and for all $M>0$, there exists $c_{\eqref{eq:pt2-unif}}(M)>0$ such that for all $m_\sigma,m_\omega \geq 1$ and for every solution $(\Psi, \overline{S},\overline{\omega})\in \mathcal{C}([-T_1, T_2],H^s)$ to~\eqref{eq:Phidef} 
which satisfies 
\begin{equation*}
 \|(\Psi, \overline{S},\overline{\omega})\|_{\mathcal{C}([-T_1,T_2],W^{1,\infty})} \leq M
\end{equation*}
the inequality
\begin{equation}\label{eq:pt2-unif}
 \|(\Psi(t), \overline{S}(t), \overline{\omega}(t))\|_{H^s} \leq R_0 e^{c_{\eqref{eq:pt2-unif}}(M) |t|}.
\end{equation}
holds for all $t\in [-T_1, T_2]$.
\end{lemma}

 \begin{proof}
 Let us remark first that for $\mu\geq 1$ and $s'\in[s-1,s]$,
\begin{equation}\label{eq.inegOp}
\bigg\|\frac{1}{\sqrt{-\Delta+\mu^2}}\bigg\|_{H^{s-1}\to H^{s'}} \leq \frac{1}{\mu^{s-s'}}\,.
\end{equation}
Let us again denote $\chi := (\Psi,\overline{\omega},\overline{S})\in \mathcal{C}([-T_1,T_2],{H}^s)$. By assumption, it satisfies
\[
\|\chi\|_{\mathcal{C}([-T_1,T_2],W^{1,\infty})}
\leq M\,.
\]
The $H^s$-norm of the terms in $\chi$ (solution of the integral equations~\eqref{eq:Phidef}) that depend only on the initial condition $\chi_\mathrm{in}:=(\Psi_\mathrm{in}, \overline{\omega}_\mathrm{in},\overline{S}_\mathrm{in})$ is bounded, due to inequality \eqref{eq.inegOp} with $s'=s$, by
\begin{align*}
 &\|\Psi_\mathrm{in}\|_{H^s}
+
\|(\overline{\omega}_\mathrm{in},\overline{S}_\mathrm{in}) \|_{H^s} 
+
 \|Q_\sigma(\chi_\mathrm{in})\|_{H^{s-1}} +  \|Q_\omega(\chi_\mathrm{in})\|_{H^{s-1}}
+
 \|(\dot{\overline{\omega}}_\mathrm{in},\dot{\overline{S}}_\mathrm{in})\|_{H^{s-1}}=R_0\,.
\end{align*}
By~\eqref{eq:PQ-bound} and \eqref{eq.inegOp} with $s'=s$ we then have
\begin{align*}
 \bigg\| \int_0^t\frac{\sin\left((t-t')\sqrt{-\Delta + m^2_\sigma}\right)}{\sqrt{-\Delta + m^2_\sigma}}P(\chi, \nabla\chi)(t')\d t'\bigg\|_{H^s}  
  &\leq 
  \int_0^{|t|}\left\| P(\chi, \nabla\chi)(t')\right\|_{H^{s-1}} \d t' \\
  &\leq c_{\eqref{eq:PQ-bound}}(M) \int_0^{|t|} \|\chi(t')\|_{H^s}\d t',
\end{align*}
and a similar bound for the term involving $Q$. By the Kato-Ponce inequality~\eqref{eq:K-P}, this gives a bound on the $H^s$- norm of the terms involving $\chi$ by
\[\begin{split}
&\int_0^{|t|}
\left(
	\|W(\chi)\Psi\|_{H^s}+ 4 c_{\eqref{eq:PQ-bound}} (M)\|\chi(t')\|_{H^s} \right)\d t'
\\&
\leq
 c_{\eqref{eq:pt2-unif}}(M)\int_0^{|t|}
\|\chi(t')\|_{{H}^s}\d t'\,,
\end{split}\]
which determines $c_{\eqref{eq:pt2-unif}}(M)>0$, and yields \eqref{eq:pt2-unif} by Gronwall's lemma.
\end{proof}
 
 \begin{remark}
   The key point of Lemma \ref{lem:unif-exist} is the uniformity in $m_\sigma$ and $m_\omega$ of the constant $c_{\eqref{eq:pt2-unif}}$, provided we have a uniform control of the solution in $W^{1,\ii}$.
 \end{remark}

\subsection{Convergence}
In this section we fix $\gamma_\sigma, \gamma_\omega\geq 0$ and $(\Psi_\mathrm{in}, S_\mathrm{in}, \dot S_\mathrm{in}, \ \omega_\mathrm{in}, \dot \omega_\mathrm{in})$. 
To prove the convergence of $\Psi$, a solution to~\eqref{eqRMFTgeneral}, to the solution $\Psi_\mathrm{nl}$ of the non-linear Dirac equation, we will first separate  $\overline{S}, \overline{\omega}$ into a ``small'' part and an ``oscillatory'' part.

\begin{lemma}\label{lem:Phi-error}
 Let $s>\tfrac52$ and $T_1,T_2>0$.  For all $s'\in[s-1,s]$ and $M>0$, there exists $c_{\eqref{eq:Phi-error}}(M)>0$ so that,
 for all $m_\sigma, m_\omega\geq 1$ and every solution $(\Psi, \overline{S}, \overline{\omega})\in\cC([-T_1,T_2],H^s)$ to~\eqref{eq:Phidef} which satisfies
 \begin{equation*}
  \|(\Psi, \overline{S}, \overline{\omega})\|_{\mathcal{C}([-T_1,T_2],H^s)}\leq M,
 \end{equation*}
we have for all $t\in [-T_1,T_2]$,
\begin{equation}\label{eq:Phi-error}
\begin{split}
&\left\|\overline{S}(t) - \cos\left(t\sqrt{-\Delta + m^2_\sigma}\right)\overline{S}_\mathrm{in}  
- \int_0^t\cos\left((t-t')\sqrt{-\Delta + m^2_\sigma}\right) Q_\sigma(\Psi, \overline{S}, \overline{\omega})(t')\d t'\right\|_{H^{s'}}
\\&
\qquad
\leq
c_{\eqref{eq:Phi-error}}(M)m_\sigma^{s'-s}\\
&\text {and} \\
&\left\|\overline{\omega}(t)-\cos\left(t\sqrt{-\Delta + m^2_\omega}\right)\overline\omega_\mathrm{in} 
- \int_0^t\cos\left((t-t')\sqrt{-\Delta + m^2_\omega}\right) Q_\omega(\Psi, \overline{S}, \overline{\omega})(t')\d t'\right\|_{H^{s'}}
\\&
\qquad
\leq
c_{\eqref{eq:Phi-error}}(M)m_\omega^{s'-s}\,.
\end{split} 
\end{equation}
\end{lemma}

\begin{proof}
Using again the notation $\chi:=(\Psi, \overline{S}, \overline{\omega})$ the formula \eqref{eq:Phidef} implies that
\begin{align*}
&\left\|\overline{\omega}(t)-\cos\left(t\sqrt{-\Delta + m^2_\omega}\right)\overline\omega_\mathrm{in}
- \int_0^t\cos\left((t-t')\sqrt{-\Delta + m^2_\omega}\right) Q_\omega(\chi)(t')\d t'\right\|_{H^{s'}}\\
&\leq
\bigg\|
\frac{\sin\left(t\sqrt{-\Delta + m^2_\omega}\right)}{\sqrt{-\Delta + m^2_\omega}}\left(\dot{\overline{\omega}}_\mathrm{in}- Q_\omega(\Psi_\mathrm{in}, \overline{\omega}_\mathrm{in}, \overline{S}_\mathrm{in} )\right)\bigg\|_{H^{s'}}\\
&\quad+
\int_0^{|t|}
\bigg\|
\frac{\sin\left((t-t')\sqrt{-\Delta + m^2_\omega}\right)}{\sqrt{-\Delta + m^2_\omega}}P_\omega(\chi, \nabla \chi)(t')
\bigg\|_{H^{s'}}
\d t' \\
&\leq
\bigg\|\frac{1}{\sqrt{-\Delta+m_\omega^2}}\bigg\|_{H^{s-1}\to H^{s'}}
\begin{aligned}[t]\Bigg(
\left\|
\dot{\overline{\omega}}_\mathrm{in}\right\|_{H^{s-1}}
&+\|Q_\omega(\Psi_\mathrm{in}, \overline{\omega}_\mathrm{in}, \overline{S}_\mathrm{in} )\|_{{H}^{s-1}}\\
&+\int_0^{|t|}
\left\|P_\omega(\chi, \nabla \chi)(t')\right\|_{H^{s-1}}
\d t'
\Bigg)\,.
\end{aligned}
\end{align*}
The claim for $\overline{\omega}$ then follows from Lemma~\ref{lem.red-eq} and the inequality \eqref{eq.inegOp}. The proof for $\overline{S}$ is the same.

\end{proof}

\begin{proposition}\label{prop.conv}
Let $s>\tfrac52$, $0<T_1<T_{\mathrm{min}}^{\mathrm{nl}}$, and $0<T_2<T_{\mathrm{max}}^{\mathrm{nl}}$.  For all $s'\in[s-1,s]$ and $M>0$, there exists $c_{\eqref{eq:Psi-loc-conv}}(M)>0$ so that 
the following holds. 
For any $m_\sigma, m_\omega\geq 1$ and every solution $(\Psi, \overline{S}, \overline{\omega})\in\cC([-T_1,T_2],H^s)$ to~\eqref{eq:Phidef} which satisfies
 \begin{equation*}
  \|(\Psi, \overline{S}, \overline{\omega})\|_{\mathcal{C}([-T_1,T_2],H^s)}\leq M,
 \end{equation*}
 we have
\begin{equation}\label{eq:Psi-loc-conv}
\|\Psi_{{}}-\Psi_{\mathrm{nl}}\|_{\mathcal{C}([-T_1,T_2],H^{s'})}\leq c_{\eqref{eq:Psi-loc-conv}}(M)\left(m_\sigma^{s'-s}+m_{\omega}^{s'-s} \right)\,. 
\end{equation}
\end{proposition}
\begin{proof}
Let $t\in [-T_1,T_2]$. We denote
\begin{equation}\label{eq:tilde-fields}
 \begin{split}
  \tilde S(t)&:=\overline{S}(t)- \cos\left(t\sqrt{-\Delta + m^2_\sigma}\right)\overline{S}_\mathrm{in} 
  -\int_0^t\cos\left((t-t')\sqrt{-\Delta + m^2_\sigma}\right) Q_\sigma(\Psi,\overline{S}, \overline{\omega})(t')\d t'\,,\\
 \tilde \omega(t)&:=\overline\omega(t)-\cos\left(t\sqrt{-\Delta + m^2_\omega}\right)\overline\omega_\mathrm{in}
 -\int_0^t\cos\left((t-t')\sqrt{-\Delta + m^2_\omega}\right) Q_\omega(\Psi,\overline{S}, \overline{\omega})(t')\d t'\,,
 \end{split}
\end{equation}
and $\tilde \omega =(\tilde V,  \tilde{\bm\omega})$.
We can then express $W(\Psi, \overline{S}, \overline{\omega})$ (defined in~\eqref{eq:W-def}) as a function of $\tilde S$, $\tilde \omega$ (on which we have a bound from Lemma~\ref{lem:Phi-error}) and $\overline{S}-\tilde{S}$, $\overline{\omega}-\tilde{\omega}$. Explicitly, from
\begin{align*}
\Psi(t)
&= e^{-itD}\Psi_\mathrm{in}
-i\int_0^te^{-i(t-t')D}W(\Psi, \overline{S}, \overline{\omega})(t')\Psi(t')\d t'\,,
\\
\Psi_{\mathrm{nl}}(t)
&\begin{aligned}[t]
=&\, e^{-itD}\Psi_\mathrm{in}\\
&-i\int_0^te^{-i(t-t')D}\left(-\gamma_\omega{\bm\alpha}\cdot {\bm J}(\Psi_\mathrm{nl}(t'))-\gamma_\sigma\beta  \rho_s(\Psi_{\mathrm{nl}}(t'))+\gamma_\omega\rho_v(\Psi_{\mathrm{nl}}(t'))\right)\Psi_{\mathrm{nl}}(t')\d t'\,,
\end{aligned}
\end{align*}
we obtain
\begin{equation}\label{eq:Psi diff}
 \|\Psi_{}(t) -\Psi_{\mathrm{nl}}(t) \|_{H^{s'}}\leq \sum_{k=1}^4\left\|\int_0^te^{-i(t-t')D}I_k (t')\d t'\right\|_{H^{s'}}\,,
\end{equation}
with
\[\begin{split}
I_1&= \left(
	\gamma_\omega 
		 {\bm\alpha}\cdot \big( {\bm J}(\Psi_\mathrm{nl}) -{\bm J}(\Psi)\big) 
		 +\gamma_\omega\big(\rho_v(\Psi) - \rho_v(\Psi_{\mathrm{nl}})\big)
		-\gamma_\sigma \beta\big(\rho_s(\Psi) - \rho_s(\Psi_{\mathrm{nl}})\big)
\right)\Psi\,,\\
I_2&= \left(-\gamma_\omega{\bm\alpha}\cdot {\bm J}(\Psi_\mathrm{nl})-\gamma_\sigma\beta  \rho_s(\Psi_{\mathrm{nl}})+\gamma_\omega\rho_v(\Psi_{\mathrm{nl}})\right)(\Psi-\Psi_\mathrm{nl}) \,,
\\
I_3&=
\left(-\bm\alpha\cdot \widetilde{\bm \omega_{}}+\beta \widetilde{S}+\widetilde{V}\right)\Psi\,,
\\
I_4&=\left(\bm\alpha\cdot (\widetilde{\bm \omega}-\overline{{\bm \omega}})+\beta(\overline{S}-\widetilde{S})+\overline{V}-\widetilde{V}\right)\Psi\,.
\end{split}\]
In the following, we repeatedly use the fact that $H^{s'}$ is a norm algebra, since $s'\geq s-1>\tfrac32$.
For the first term, using that ${\bm J}, \rho_v, \rho_s$ are quadratic expressions in $\Psi$, respectively $\Psi_\mathrm{nl}$, this gives 
\begin{equation}\label{eq:I1}
 \left\|\int_0^te^{-i(t-t')D}I_1(t')\d t'\right\|_{H^{s'}}\leq c_{\eqref{eq:I1}}(M)\int_0^{|t|}\|\Psi(t')-\Psi_{\mathrm{nl}}(t')\|_{H^{s'}}\d t'\,,
\end{equation}
for some $c_{\eqref{eq:I1}}(M)>0$,  due to the uniform $H^{s}$-bound on $\Psi$ (we also used that $\Psi_{\mathrm{nl}}$ is uniformly bounded in $H^s$ by a constant which only depends on $M$, since $[-T_1,T_2]\subset(-T_{\mathrm{min}}^{\mathrm{nl}},T_{\mathrm{max}}^{\mathrm{nl}})$ and since $\|\Psi_{\mathrm{in}}\|_{H^s}\leq M$). 
Similarly for $I_2$ we have
\begin{equation}\label{eq:I2}
 \left\|\int_0^te^{-i(t-t')D}I_2(t')\d t'\right\|_{H^{s'}} \leq c_{\eqref{eq:I2}}(M) \int_0^{|t|} \|\Psi_{}(t')-\Psi_{\mathrm{nl}}(t')\|_{H^{s'}} \d t'.
\end{equation}
 Hence these two terms are bounded by a an expression involving $\|\Psi_{}-\Psi_{\mathrm{nl}}\|_{H^{s'}}$ which can be controlled by Gronwall's Lemma.

By Lemma~\ref{lem:Phi-error}, there exists $c_{\eqref{eq:I3}}(M)>0$ so that
\begin{equation}\label{eq:I3}
 \left\|\int_0^te^{-i(t-t')D}I_3(t')\d t'\right\|_{H^{s'}}\leq c_{\eqref{eq:I3}}(M)\left(m_\sigma^{s'-s}+m_{\omega}^{s'-s}\right)\,.
\end{equation}
This term thus satisfies the stated bound and converges to zero as $m_\sigma, m_\omega\to \infty$ if $s'<s$. 

The final term $I_4$ is given by an integral of $\Psi$ with the oscillatory contributions to $\overline{S}, \overline{\omega}$. We will prove the bound using integration by parts. We give the details only for the term $\beta(\overline{S}-\widetilde{S})$, as the argument for the other terms is the same.
Recall that
\begin{equation}
 \overline{S}(t)-\widetilde{S}(t) =\cos\left(t\sqrt{-\Delta + m^2_\sigma}\right)\overline{S}_\mathrm{in} 
  +\int_0^t\cos\left((t-t')\sqrt{-\Delta + m^2_\sigma}\right) Q_\sigma(\Psi,\overline{S}, \overline{\omega})(t')\d t'\,.
\end{equation}
For the first contribution, we obtain using the equation satisfied by $\Psi$
\begin{align*}
 &\int_0^t e^{-i(t-t')D}\beta \left(\cos(t'\sqrt{-\Delta + m^2_\sigma})\overline{S}_\mathrm{in}\right)\Psi(t') \d t' \\
 &=\beta\left(\frac{\sin\left(
t\sqrt{-\Delta+m_\sigma^2}
\right)}{\sqrt{-\Delta+m_\sigma^2}}\overline{S}_\mathrm{in}\right)\Psi_{}(t)
\\&\quad
-i
\int_0^te^{-i(t-t')D}D\left(
\beta\frac{\sin (t'\sqrt{-\Delta+m_\sigma^2})}{\sqrt{-\Delta+m_\sigma^2}}\overline{S}_\mathrm{in}\right)
\Psi(t')\d t'
\\&\quad+ i
\int_0^te^{-i(t-t')D}\beta\left(\frac{\sin\left(
t'\sqrt{-\Delta+m_\sigma^2}
\right)}{\sqrt{-\Delta+m_\sigma^2}}\overline{S}_\mathrm{in}\right)(D+W(\Psi, \overline{S}, \overline \omega))\Psi(t')\d t'\,.
\end{align*}
We have
\begin{equation}\label{eq:I4-prep-s-1}
 \left\| D\left(
\beta\frac{\sin (t'\sqrt{-\Delta+m_\sigma^2})}{\sqrt{-\Delta+m_\sigma^2}}\overline{S}_\mathrm{in}\right)
\Psi(t') \right\|_{H^{s-1}} \leq c_{\eqref{eq:I4-prep-s-1}} m_\sigma^{-1} M,
\end{equation}
where $c_{\eqref{eq:I4-prep-s-1}}$ depends only on $\overline{S}_\mathrm{in}$. With a similar bound on the other terms, we obtain
\begin{equation}\label{eq:I4 s-1}
 \left\| \int_0^t e^{-i(t-t')D}\beta \left(\cos(t'\sqrt{-\Delta + m^2_\sigma})\overline{S}_\mathrm{in}\right)\Psi(t') \d t'\right\|_{H^{s-1}}
 \leq c_{\eqref{eq:I4 s-1}}(M) m_\sigma^{-1}.
\end{equation}
Obviously we also have, without integrating by parts,
\begin{equation}\label{eq:I4 s}
 \left\| \int_0^t e^{-i(t-t')D}\beta \left(\cos(t'\sqrt{-\Delta + m^2_\sigma})\overline{S}_\mathrm{in}\right)\Psi(t') \d t'\right\|_{H^{s}}
 \leq c_{\eqref{eq:I4 s}}  (M).
\end{equation}
Hence, by interpolation, 
\begin{align}\label{eq:I4 s'}
 \bigg\| \int_0^t e^{-i(t-t')D}\beta \Big(\cos(t'\sqrt{-\Delta + m^2_\sigma})\overline{S}_\mathrm{in}\Big)&\Psi(t') \d t'\bigg\|_{H^{s'}} \notag  \\
 &\leq  c_{\eqref{eq:I4 s-1}}(M)^{s-s'} c_{\eqref{eq:I4 s}}(M)^{s'+1-s} m_\sigma^{s'-s}.
\end{align}
This is the desired estimate for the the term coming from $\overline{S}_\mathrm{in}$.  
We proceed similarly for the other term. 
Integrating by parts in the integral over $t'$, we find
\begin{align*}
 &\int\limits_0^t e^{-i(t-t')D}\int\limits_0^{t'}  \left(\beta\cos\left((t'-t'')\sqrt{-\Delta + m^2_\sigma}\right) Q_\sigma(\Psi,\overline{S}, \overline{\omega})(t'')\right)\d t'' \Psi(t')\d t' \\
 &= -i \int\limits_0^t \int\limits_0^{t'} e^{-i(t-t')D} D \left( \beta \frac{\sin((t'-t'')\sqrt{-\Delta + m^2_\sigma})}{\sqrt{-\Delta + m^2_\sigma}} Q_\sigma(\Psi,\overline{S}, \overline{\omega})(t'')\right) \Psi(t')\d t''  \d t'\\
 &\, + i  \int\limits_0^t \int\limits_0^{t'} e^{-i(t-t')D}\hspace{-1pt}\left(\frac{\sin((t'-t'')\sqrt{-\Delta + m^2_\sigma})}{\sqrt{-\Delta + m^2_\sigma}} Q_\sigma(\Psi,\overline{S}, \overline{\omega})(t'')\right)\hspace{-1pt} (D+W(t'))\beta\Psi(t')\d t'' \d t'.
\end{align*}
From this we obtain a similar estimate to~\eqref{eq:I4 s'} by interpolation.
Repeating the same argument for the terms coming from $\overline{\omega}-\tilde{\omega}$, we obtain
\begin{equation}\label{eq:I4}
  \left\|\int_0^te^{-i(t-t')D}I_4(t')\d t'\right\|_{H^{s'}}\leq c_{\eqref{eq:I4}}(M) (m_\sigma^{s'-s} + m_\omega^{s'-s}).
\end{equation}
To complete the proof, we apply Gronwall's inequality to~\eqref{eq:Psi diff}, which gives 
\begin{equation}
 \|\Psi_{{}}-\Psi_{\mathrm{nl}}\|_{H^{s'}}
\leq (c_{\eqref{eq:I3}}(M) + c_{\eqref{eq:I4}}(M))e^{(c_{\eqref{eq:I1}}(M) +c_{\eqref{eq:I2}}) |t|} \left(m_\sigma^{s'-s}+m_{\omega}^{s'-s}\right),
\end{equation}
and thus the claimed inequality~\eqref{eq:Psi-loc-conv}, since $|t|\leq \max\{T_1,T_2\}$.

\end{proof}
\subsection*{Proof of Theorem~\ref{thm.main1}}

Let $\tfrac52 <s'<s$, $0<T_1<T_{\mathrm{min}}^{\mathrm{nl}}$, $0<T_2<T_{\mathrm{max}}^{\mathrm{nl}}$. Our first goal is to show that there exists $\mu>0$ such that for all $m_\omega,m_\sigma\geq \mu$, we have $T_{\mathrm{min}}>T_1$ and $T_{\mathrm{max}}>T_2$. This indeed implies that
\begin{equation*}
\liminf_{m_\sigma, m_\omega\to \infty}T_{\rm min/max}\geq T_{\rm min/max}^{\mathrm{nl}}\,.
\end{equation*}
We prove this property for $T_{\mathrm{max}}$, the proof for $T_{\mathrm{min}}$ being identical. To do so, let $M>\|(\Psi_\mathrm{in}, \overline{S}_\mathrm{in}, \overline{\omega}_\mathrm{in})\|_{W^{1,\infty}}$, to be chosen later (and which will depend only on $s'$ and the initial data) and define 
$$T_2' = \sup\{0<t<\min(T_2,T_{\mathrm{max}}),\,\|(\Psi,\overline{S}, \overline{\omega})\|_{\mathcal{C}([0,t], W^{1,\infty})}\leq M\}\,.$$
By definition of $T_2'$, we have 
\begin{equation}\label{eq:bound-T2prime}
\sup_{0\leq t<T_2'}\|(\Psi,\overline{S}, \overline{\omega})(t)\|_{W^{1,\infty}}\leq M. 
\end{equation}
By the blow-up criterion \eqref{eq:blowup-crit}, we have $T_2'<T_{\mathrm{max}}$. Indeed, if $T_{\mathrm{max}}=+\ii$ this is clear (because $T_2'\leq T_2<+\ii$), and if $T_{\mathrm{max}}<+\ii$ the criterion \eqref{eq:blowup-crit} and the relation between $(S,\omega)$ and $(\bar{S},\bar{\omega})$ imply that 
$$\sup_{0\leq t<T_{\mathrm{max}}}\|(\Psi,\overline{S}, \overline{\omega})(t)\|_{W^{1,\infty}}=+\ii$$
which shows that $T_2'\neq T_{\mathrm{max}}$ and thus $T_2'<T_{\mathrm{max}}$ (since we always have $T_2'\leq T_{\mathrm{max}}$ by definition of $T_2'$). To prove that $T_{\mathrm{max}}>T_2$ when $m_\sigma,m_\omega$ are large enough, we will thus prove that, for a good choice of $M$, there exists $\mu>0$ such that for all $m_\omega,m_\sigma\geq \mu$, we have $T_2=T_2'$ (recall that we always have $T_2'\leq T_2$ by definition of $T_2'$). We do so by contradiction, assuming that for all $\mu>0$ there are $m_\omega,m_\sigma\geq\mu$ such that $T_2'<T_2$. Since $T_2'<T_2$, notice that we have 
\begin{equation}\label{eq:bound-T2prime-2}
 \sup_{0\leq t\leq T_2'} \|(\Psi,\overline{S}, \overline{\omega})(t)\|_{W^{1,\infty}}=M.
\end{equation}
Indeed, since $T_2'<T_{\mathrm{max}}$ we know that $(\Psi,\overline{S}, \overline{\omega})\in\cC([0,T_2'],W^{1,\infty})$ and together with \eqref{eq:bound-T2prime} this implies
$$
\sup_{0\leq t\leq T_2'}\|(\Psi,\overline{S}, \overline{\omega})(t)\|_{W^{1,\infty}}\leq M. 
$$
Now if we had 
$$
\sup_{0\leq t\leq T_2'}\|(\Psi,\overline{S}, \overline{\omega})(t)\|_{W^{1,\infty}}<M,
$$
we could extend this inequality on $[0,T]$ for some $T>T_2'$ with $T<T_2$ and $T<T_{\mathrm{max}}$, using that $(\Psi,\overline{S}, \overline{\omega})\in\cC([0,T_{\mathrm{max}}),W^{1,\infty})$ and that $T_2'<T_2$, $T_2'<T_{\mathrm{max}}$, in contradiction with the maximality of $T_2'$. Hence, we indeed have \eqref{eq:bound-T2prime-2}. 

Next, by \eqref{eq:pt2-unif} and \eqref{eq:bound-T2prime-2}, there exists $M'>0$ (which only depends on $M$, on $T_2$, and on the initial data) so that
\[
\|(\Psi,\overline{\omega},\overline{S})\|_{\mathcal{C}([0,T_2'], {H}^s)}
\leq 
M'\,,
\]
By Lemma~\ref{lem:Phi-error} and Lemma~\ref{lem.red-eq} we then have for $t\in [0, T_2']$
\begin{align*}
 \|(\overline{S}(t), \overline{\omega}(t))\|_{H^{s'}}
 &\leq  2\mu^{s'-s} c_{\eqref{eq:Phi-error}}(M') + \|(\overline{S}_\mathrm{in},\overline{\omega}_\mathrm{in} )\|_{H^{s'}} \\
 &\quad + \left\|\int_0^t\cos\left((t-t')\sqrt{-\Delta + m^2_\sigma}\right) Q(\Psi,  \overline{S}, \overline{\omega})(t')\d t'\right\|_{H^{s'}} \\
 &\leq  2\mu^{s'-s} c_{\eqref{eq:Phi-error}}(M') + \|(\overline{S}_\mathrm{in},\overline{\omega}_\mathrm{in} )\|_{H^{s'}} \\
 &\quad + c_{\eqref{eq:Q-bound}}\int_0^{t} \left(\|\Psi(t')\|_{H^{s'}}^2 \|(\overline{S}(t'), \overline{\omega}(t'))\|_{H^{s'}} + \|\Psi(t')\|_{H^{s'}}^4\right)\d t.
\end{align*}
By Gronwall's inequality this yields
\begin{align}
 &\|(\overline{S}, \overline{\omega})\|_{\mathcal{C}([0,T_2'], H^{s'})} \notag\\
  &\leq \left( 2\mu^{s'-s}c_{\eqref{eq:Phi-error}}(M')  + \|(\overline\omega_\mathrm{in}, \overline{S}_\mathrm{in} )\|_{H^{s'}} +c_{\eqref{eq:Q-bound}}  T_2 \|\Psi\|_{\mathcal{C}([0,T_2'], H^{s'})}^4 \right)
e^{c_{\eqref{eq:Q-bound}} T_2 \|\Psi\|_{\mathcal{C}([0,T_2'], H^{s'})}^2}.\label{eq:S-Psi bound}
\end{align}
Hence the growth of $\overline{S}, \overline{\omega}$ is controlled by that of $\Psi$, which in turn will be close to $\Psi_\mathrm{nl}$, the solution of \eqref{eq.dnl1}.
Let $f$ be the function
\begin{equation*}
 f(x)=x + (\|(\overline\omega_\mathrm{in}, \overline{S}_\mathrm{in} )\|_{H^{s'}} + c_{\eqref{eq:Q-bound}} T_2 x^4) e^{c_{\eqref{eq:Q-bound}} T_2 x^2}
\end{equation*}
of the real variable $x$. Note that this function does not depend on $M$. Then 
 \begin{align*}
   \|(\Psi, \overline{S},\overline{\omega})\|_{\mathcal{C}([0,T_2'], {H}^{s'})}
   &\leq \|\Psi\|_{\mathcal{C}([0,T_2'], {H}^{s'})} + \|(\overline{S}, \overline{\omega})\|_{\mathcal{C}([0,T_2'], H^{s'})} \\
   &\leq f(\|\Psi\|_{\mathcal{C}([0,T_2'], {H}^{s'})}) + 2\mu^{s'-s} c_{\eqref{eq:Phi-error}}(M') e^{c_{\eqref{eq:Q-bound}} T_2 \|\Psi\|_{\mathcal{C}([0,T_2'], H^{s'})}^2}\,.\\
 \end{align*}
Defining 
\begin{equation*}
M= 2 C_{s'} f(\|\Psi_{\mathrm{nl}}\|_{\mathcal{C}([0,T_2], {H}^{s'})}),
\end{equation*}
 where $C_{s'}$ denotes the norm of the Sobolev embedding ${H}^{s'}\hookrightarrow W^{1,\infty}$, we have 
 \begin{multline}\label{eq:boundthm1}
   \|(\Psi, \overline{S},\overline{\omega})\|_{\mathcal{C}([0,T_2'], W^{1,\ii)}} \leq \tfrac12 M + C_{s'}\Big(f(\|\Psi\|_{\mathcal{C}([0,T_2'], {H}^{s'})})-f(\|\Psi_{\mathrm{nl}}\|_{\mathcal{C}([0,T_2'], {H}^{s'})})\Big) \\
   + 2C_{s'} \mu^{s'-s} c_{\eqref{eq:Phi-error}}(M') e^{c_{\eqref{eq:Q-bound}} T_2 \|\Psi\|_{\mathcal{C}([0,T_2'], H^{s'})}^2}
 \end{multline}
By Proposition~\ref{prop.conv}, we have
\begin{equation}\label{eq:Psi-diff}
 \|\Psi-\Psi_{\mathrm{nl}}\|_{\mathcal{C}([0, T_2'],H^{s'})}\leq c_{\eqref{eq:Psi-loc-conv}}(M') 2\mu^{s-s'}\,.
\end{equation}
 In particular, $\|\Psi\| \to \|\Psi_\mathrm{nl}\|$ as $\mu\to\infty$. Using this fact and the continuity of $f$, we deduce from \eqref{eq:boundthm1} that for $\mu$ large enough we have 
 $$\|(\Psi, \overline{S},\overline{\omega})\|_{\mathcal{C}([0,T_2'], W^{1,\ii)}}<M,$$
 a contradiction with \eqref{eq:bound-T2prime-2}. We thus have proved that for this choice of $M$ and for $\mu$ large enough, we have $T_2'=T_2$. The convergence of $\Psi$ towards $\Psi_{\mathrm{nl}}$ on $\cC([0,T_2],H^{s'})$ then follows from \eqref{eq:Psi-diff}.

\qed

\begin{remark}\label{rk:rate}
For any $s>\tfrac52$, any $s'\in\left(\tfrac52,s\right]$, any $(\Psi_{\mathrm{in}},S_{\mathrm{in}}, \dot{S}_{\mathrm{in}}, \omega_{\mathrm{in}},\dot{\omega}_{\mathrm{in}})$, any $T_1\in(0,T_{\mathrm{min}}^{\mathrm{nl}})$ and any $T_2\in(0,T_{\mathrm{max}}^{\mathrm{nl}})$, the above proof shows that there exists $c>0$ and $\mu>0$ such that for all $m_\sigma,m_\omega\geq\mu$ we have
$$\sup_{-T_1\leq t\leq T_2}\|\Psi(t)-\Psi_{\mathrm{nl}}(t)\|_{H^{s'}}\le c(m_\sigma^{-r}+m_\omega^{-r}),$$
with $r=\min\{1,s-s'\}$.
We expect the rate of convergence $s-s'$ to be sharp if the initial data is no better than $H^s$ (as can be seen from \eqref{eq:I4-prep-s-1} for instance), while $r\le1$ is necessary for generic initial data.
\end{remark} 

 \section{The many-body case}\label{sect:many-body}
 
 In this section we prove Theorem~\ref{thm.main-mb}. The structure of the proof, and many arguments, are very similar to the one-body case of Theorem~\ref{thm.main1}. We will thus not give the details in proofs that remain essentially unchanged and focus more on the points where the case of density matrices requires additional care.
 
 \subsection{Functional setting}
 
 For $s\geq 0$, we say that $\Gamma\in\gH^s$ if $\Gamma$ is a bounded operator from $H^{-s}(\R^3,\C^4)$ to $H^s(\R^3,\C^4)$ and if the operator $(1-\Delta)^{s/2}\Gamma (1-\Delta)^{s/2}$  (which is a well-defined bounded operator on $L^2$) is Hilbert-Schmidt on $L^2$. This space $\gH^s$ is then a Banach space endowed with the norm \eqref{eq:norm-H^s-op}. Notice that $\gH^0=\gS^2$ and that $\gH^{s}\hookrightarrow \gH^{s'}$ for any $0\leq s'\leq s$. Below, we will come across operators such as $(1-\Delta)^{s/2}\Gamma(1-\Delta)^{s'/2}$ and $(1-\Delta)^{s'/2}\Gamma(1-\Delta)^{s/2}$ which for $\Gamma\in\gH^{s}$ and $0\leq s'\leq s$ are also Hilbert-Schmidt operators. Recall that any Hilbert-Schmidt operator $\Gamma$ on $L^2(\R^3,\C^4)$ has an integral kernel which belongs to $L^2(\R^3\times\R^3,\C^{4\times4})$ and that we denote by $(x,y)\mapsto\Gamma(x,y)$.
 
 For $\Gamma\in\gH^s$, the densities $\rho_s(\Gamma)$, $\rho_v(\Gamma)$, and $J(\Gamma)$ defined in \eqref{eq:rho-s-def}, \eqref{eq:rho-v-def}, and \eqref{eq:rho-J-def} are regular by the following lemma.
 
 \begin{lemma}\label{lem:est-density}
 Let $s>\tfrac32$. Then, there exists $c_\eqref{eq:est-density-s}>0$ such that for any $\Gamma \in \mathfrak{H}^s$, $x\mapsto \Gamma(x,x)\in H^s(\R^3, \C^{4\times 4})$, and we have
 \begin{equation}\label{eq:est-density-s}
  \|\Gamma(x,x)\|_{H^s(\R^3, \C^4)}\leq c_{\eqref{eq:est-density-s}} \|\Gamma\|_{\mathfrak{H}^s}.
  \end{equation}
  Moreover, if $s'\in\left(\tfrac32,s\right)$ then there exists $c_\eqref{eq:est-density}>0$ such that for any $\Gamma \in \mathfrak{H}^s$ which is additionally non-negative we have
 \begin{equation}\label{eq:est-density}
  \|\Gamma(x,x)\|_{H^s(\R^3, \C^4)}\leq c_{\eqref{eq:est-density}} \|\Gamma\|_{\mathfrak{H}^s}^{1/2}\|\Gamma\|_{\mathfrak{H}^{s'}}^{1/2}.
\end{equation}
\end{lemma}
 
 \begin{proof}[Proof of Lemma \ref{lem:est-density}]
 The assumption $\Gamma\in \mathfrak{H}^s$ means that 
 $$\int_{\R^3}\int_{\R^3}(1+|p|^2)^s(1+|q|^2)^s\|\hat{\Gamma}(p,q)\|^2_{\C^{4\times 4}}\,dp\,dq<+\ii,$$
 where $\hat{\Gamma}$ denotes the integral kernel of $\cF\Gamma\cF^*$ (and $\cF$ is the Fourier transform on $L^2$). 
 This implies that $\hat{\Gamma}\in L^1(\R^3\times\R^3, \C^{4\times 4})$ and hence $\Gamma(\cdot,\cdot)\in \mathcal{C}(\R^3\times\R^3, \C^{4\times 4})$. Furthermore, we have for all $x\in\R^3$,
 $$\Gamma(x,x)=\frac{1}{(2\pi)^3}\int_{\R^3}\int_{\R^3}\hat{\Gamma}(p,q)e^{i(p-q)\cdot x}\,dp\,dq,$$
 hence by the Plancherel identity we deduce that 
 $$\|\Gamma(x,x)\|_{H^s}^2 = \frac{1}{(2\pi)^3}\int_{\R^3} (1+|k|^2)^s \left\|\int_{\R^3}\hat{\Gamma}(p,k-p)\,dp\right\|^2_{\C^{4\times 4}}\,dk.$$
 In view of the bound
 \begin{equation}\label{eq:s-s'-bound}
\int_{\R^3}\frac{dp}{(1+|p|^2)^{s'}(1+|p-k|^2)^s + (1+|p|^2)^s(1+|p-k|^2)^{s'}}\leq
\frac{c_{\eqref{eq:s-s'-bound}}}{(1+|k|^2)^s},  
 \end{equation}
 which holds for $s\geq s'>\tfrac32$ (by decomposition of the integral into parts with $|p-k|\lessgtr\tfrac12 |k|$),
the Cauchy-Schwarz inequality implies that
\begin{equation}
 \|\Gamma(x,x)\|_{H^s}\leq \frac{\sqrt{c_{\eqref{eq:s-s'-bound}}}}{(2\pi)^{3/2}}\left(\|(1-\Delta)^{s'/2}\Gamma(1-\Delta)^{s/2}\|_{\gS^2} + \|(1-\Delta)^{s/2}\Gamma(1-\Delta)^{s'/2}\|_{\gS^2}\right).
\end{equation}
Taking $s'=s$, this proves \eqref{eq:est-density-s}.
Now for $s'<s$ and $\Gamma\geq 0$ we have by the H\"older inequality in Schatten spaces
\begin{align}
 \|(1-\Delta)^{s/2}\Gamma(1-\Delta)^{s'/2}\|_{\gS^2}^2
 &= \|(1-\Delta)^{s/2}\sqrt{\Gamma}\sqrt{\Gamma}(1-\Delta)^{s'/2}\|_{\gS^2}^2 \notag\\
 &\leq \|(1-\Delta)^{s/2}\sqrt{\Gamma}\|_{\gS^4}^2\|\sqrt{\Gamma}(1-\Delta)^{s'/2}\|_{\gS^4}^2  \notag\\
 &= \|\Gamma\|_{\mathfrak{H}^s}\|\Gamma\|_{\mathfrak{H}^{s'}},\label{eq:Gamma-mixed-bound}
\end{align}
and this completes the proof.

\end{proof}

\begin{remark}
 The above lemma shows that the restriction of $(x,y)\mapsto\Gamma(x,y)$ to the diagonal $\{y=x\}$ belongs to $H^s$ if $\Gamma\in\gH^s$. Notice here that one does not lose derivatives as in the Sobolev trace theorem, the reason being that $(x,y)\mapsto \Gamma(x,y)$ actually belongs to $H^s(\R^3)\otimes H^s(\R^3)$ which is smaller than $H^s(\R^3\times\R^3)$ (in Fourier variables, the former means that $(1+|p|^2)^{s/2}(1+|q|^2)^{s/2}\hat{\Gamma}(p,q)\in L^2$ while the latter means that $(1+|p|^2+|q|^2)^{s/2}\hat{\Gamma}(p,q)\in L^2$). Furthermore, the non-negativity assumption is necessary for \eqref{eq:est-density}: if it were true for all $\Gamma\in\gH^s$, one could take 
 $\hat{\Gamma}(p,q)=f\left(\frac{p}{\epsilon}\right)f\left(\frac{q-\xi_0}{\epsilon}\right)+f\left(\frac{q}{\epsilon}\right)f\left(\frac{p-\xi_0}{\epsilon}\right)$
 with $f\in \mathcal{C}^\ii_0(\R^3)$, $\xi_0\neq0$ and $\epsilon>0$ (the resulting $\Gamma$ is even self-adjoint), and as $\epsilon\to0$ one would deduce from \eqref{eq:est-density} that 
 $$(1+|\xi_0|^2)^s + (1+|\xi_0|^2)^{s'} \leq C(1+|\xi_0|^2)^{\frac{s+s'}{2}},$$
 which is wrong if $s'<s$ as $|\xi_0|\to+\ii$.
\end{remark}

With these preparations we can now give a precise meaning to the equations involving the operator $\Gamma$. If $I\subset\R$ is an open time interval containing $0$, $s>\tfrac32$, and $F\in \cC(I,H^s)$, we say that $\Gamma\in \cC(I,\gH^s)$ is a solution to the equation 
\begin{equation}\label{eq:gamma-example}
i\partial_t\Gamma=[D+F(t),\Gamma]
\end{equation}
if we have 
$$i\partial_t \langle f,\Gamma(t) g\rangle = \langle Df,\Gamma(t) g\rangle - \langle f,\Gamma(t) Dg\rangle + \langle f,[F(t),\Gamma(t)]g\rangle
$$
in the sense of distributions $\cD'(I)$ for all $f,g\in L^2(\R^3,\C^4)$. Notice that the assumption on $F$ shows that the multiplication operator by $F(t)$ is bounded on $L^2(\R^3,\C^4)$ for any $t\in I$. By standard arguments,  $\Gamma\in \cC(I,\gH^s)$ is a solution to \eqref{eq:gamma-example} in this sense if and only if it solves the integral equation
\begin{equation}
 \Gamma(t)= e^{-itD} \Gamma_{\mathrm{in}} e^{itD} - i\int_0^te^{-i(t-t')D}\big[F(t'),\Gamma(t')\big]e^{i(t-t')D}\,\d t'.
\end{equation}
where the integral is taken in the weak sense (i.e. $\langle f,\int_0^t B(t')\,\d t' g\rangle := \int_0^t \langle f,B(t') g\rangle \d t'$ for all $f,g\in L^2$). We will consider this integral equation with $F$ given either in terms of $(S,\omega)$ in \eqref{eq:dkg-mb} or in terms of $\Gamma$ in \eqref{eq:dnl-mb} (the corresponding $F$ belongs to $\cC(I,H^s)$ in both cases).

To prove the analog of Proposition \ref{prop:exist} for this kind of integral equation, one needs to bound the commutators or products of $\Gamma$ with multiplication operators. The required bounds in our setting are provided by the following lemma.

\begin{lemma}\label{lem:K-P-mb}
 Let $s>\tfrac32$. For all $\Gamma\in \mathfrak{H}^s$ and $f\in H^s(\R^3, \C)$, the products $\Gamma f$ and $f\Gamma$ define elements of $ \mathfrak{H}^s$. Moreover, for every $\tfrac32<s'\leq s$ there exists a constant  $c_{\eqref{eq:K-P-mb}}>0$ so that, for all $\Gamma\in \mathfrak{H}^s$ satisfying additionally $\Gamma \geq 0$ if $s'<s$ and all $f\in H^s(\R^3, \C)$, we have
 \begin{equation}\label{eq:K-P-mb}
  \begin{aligned}
  \| f\Gamma \|_{\mathfrak{H}^s} + \| \Gamma f\|_{\mathfrak{H}^s} &\leq c_{\eqref{eq:K-P-mb}} \big(\|f\|_{H^s} \|\Gamma\|_{\mathfrak{H}^s}^{1/2} \|\Gamma\|_{\mathfrak{H}^{s'}}^{1/2} + \|f\|_{L^\infty} \|\Gamma\|_{\mathfrak{H}^s} \big).
  \end{aligned}
 \end{equation}
\end{lemma}
\begin{proof}
Since $H^s$ is an algebra for $s>\tfrac32$ it is clear that mulitplication by $f$ is a bounded operator on $H^s$ and $H^{-s}$, so $f\Gamma, \Gamma f \in \mathfrak{H}^s$.
 The norm of the operator $\Gamma\in \mathfrak{H}^s$ is equal to the norm of the kernel $\|(1-\Delta_y)^{s/2}(1-\Delta_x)^{s/2}\Gamma(x,y)\|_{L^2(\R^6)}$ (since the Hilbert-Schmidt norm is the $L^2$-norm of the kernel). The kernel of the operator $f\Gamma$ is $f(x)\Gamma(x,y)$, so by the Kato-Ponce inequality~\eqref{eq:K-P}
 \begin{align*}
  \| f\Gamma \|_{\mathfrak{H}^s}^2 & =  \int_{\R^3} \|(1-\Delta_x)^{s/2}f(x)(1-\Delta_y)^{s/2}\Gamma(x,y)\|_{L^2_x}^2 \d y\\
  &\leq 2c_{\eqref{eq:K-P}}^2  \int_{\R^3}\|f\|_{H^s}^2 \|(1-\Delta_y)^{s/2}\Gamma(x,y)\|_{L^\infty_x}^2 \d y \\
  &\qquad + 2c_{\eqref{eq:K-P}}^2  \int_{\R^3}\|f\|_{L^\infty}^2\|(1-\Delta_x)^{s/2}(1-\Delta_y)^{s/2}\Gamma(x,y)\|_{L^2_x}^2 \d y \\
  &\leq 2c_{\eqref{eq:K-P}}^2 \big(\|f\|_{H^s} \|(1-\Delta_y)^{s/2}\Gamma(x,y)\|_{L^2_yL^\infty_x} + \|f\|_{L^\infty} \|\Gamma\|_{\mathfrak{H}^s} \big)^2.
 \end{align*}
 By Sobolev embedding $H^{s'}\hookrightarrow L^\infty$ (and the inequality~\eqref{eq:Gamma-mixed-bound} for $s'<s$) we obtain
 \begin{equation*}
  \|(1-\Delta_y)^{s/2}\Gamma(x,y)\|_{L^2_yL^\infty_x} \leq c \|(1-\Delta_x)^{s'/2}(1-\Delta_y)^{s/2}\Gamma(x,y)\|_{L^2_yL^2_x} \leq c \|\Gamma\|_{\mathfrak{H}^s}^{1/2} \|\Gamma\|_{\mathfrak{H}^{s'}}^{1/2}.
 \end{equation*}
The bound for $\Gamma f$ is the same up to exchange of $x, y$ in the proof, and this yields~\eqref{eq:K-P-mb}.

\end{proof}

The inequality $\eqref{eq:est-density-s}$ plays the same role in the many-body setting as the fact that $H^s$ is a normed algebra for $s>\tfrac32$ in the one-body case. The bound~\eqref{eq:K-P-mb} replaces the Kato-Ponce inequality in the many-body case (in particular, the presence of $\|\Gamma\|_{\gH^{s'}}^{1/2}$ in \eqref{eq:K-P-mb} and \eqref{eq:est-density} will play a crucial role below). 

\subsection{Well-posedness}

The integral formulations on a time interval $I$ containing $0$ of the various equations we are working on are the following. Let $s>\tfrac32$. For the many-body Dirac-Klein-Gordon system, $(\Gamma,S,\omega)\in\cC(I,\gH^s\times H^s)$ is a solution to \eqref{eq:dkg-mb} with the same initial conditions as in Theorem \ref{thm.main-mb} if and only if for all $t\in I$,
 \begin{equation}
  \begin{cases}\label{eq:dkg-mb-int}
\dps\Gamma(t)= e^{-itD} \Gamma_{\mathrm{in}} e^{itD} - i\int_0^te^{-i(t-t')D}\big[({ \beta}S +  V - {\bm \alpha}\cdot {\bm \omega})(t'),\Gamma(t')\big]e^{i(t-t')D}\,\d t'\,,
\\
\\
 \dps S(t)= 
\cos\left(t\sqrt{-\Delta + m^2_\sigma}\right) S_\mathrm{in}
+ \frac{\sin\left(t\sqrt{-\Delta + m^2_\sigma}\right)}{\sqrt{-\Delta + m^2_\sigma}}\dot{S}_\mathrm{in} 
\\\dps\qquad\qquad\qquad 
-g_\sigma^2\int_0^t\frac{\sin\left((t-t')\sqrt{-\Delta + m^2_\sigma}\right)}{\sqrt{-\Delta + m^2_\sigma}}\rho_s(\Gamma(t'))\d t',\\
\\
 \dps\omega(t)= 
\cos\left(t\sqrt{-\Delta + m^2_\omega}\right)\omega_\mathrm{in}
+ \frac{\sin\left(t\sqrt{-\Delta + m^2_\omega}\right)}{\sqrt{-\Delta + m^2_\omega}}\dot{\omega}_\mathrm{in}
\\\dps\qquad\qquad\qquad 
+g_\omega^2\int_0^t\frac{\sin\left((t-t')\sqrt{-\Delta + m^2_\omega}\right)}{\sqrt{-\Delta + m^2_\omega}}J(\Gamma(t')) \d t'.\\
\end{cases}
 \end{equation}
 For the many-body nonlinear Dirac equation, $\Gamma\in\cC(I,\gH^s)$ solves \eqref{eq:dnl-mb} if and only if for all $t\in I$
\begin{multline}\label{eq:dnl-mb-int}
 \Gamma(t)= e^{-itD} \Gamma_{\mathrm{in}} e^{itD} \\
 - i\int_0^te^{-i(t-t')D}\big[(- { \beta}\gamma_\sigma \rho_s(\Gamma) +  \gamma_\omega \rho_v(\Gamma) -  \gamma_\omega {\bm \alpha}\cdot {\bm J}(\Gamma))(t'),\Gamma(t')\big]e^{i(t-t')D}\,\d t'.
\end{multline}

From Lemma \ref{lem:est-density} and Lemma \ref{lem:K-P-mb} in the case $s'=s$, one can use a  straightforward fixed point argument on the integral formulations \eqref{eq:dkg-mb-int} and \eqref{eq:dnl-mb-int} to obtain the following analog of the existence of maximal solutions in Proposition \ref{prop:exist} (similar arguments can be found in \cite[Thm. 3]{LewSab-13a}). 

\begin{proposition}\label{prop:exist-mb}
 Let $s>\tfrac32$, $\Gamma_\mathrm{in}\in \mathfrak{H}^s$ be a non-negative operator, $(S_\mathrm{in},\dot{S}_\mathrm{in})\in H^s(\R^3, \R)\times H^{s-1}(\R^3, \R)$, and $(\omega_\mathrm{in},\dot{\omega}_\mathrm{in})\in H^s(\R^3, \R^4)\times H^{s-1}(\R^3, \R^4)$.
  \begin{enumerate}[label=(\roman*)]
  \item\label{pt1.prop-mb} There exist $T_{\mathrm{min}}^{\mathrm{nl}}, T_{\mathrm{max}}^{\mathrm{nl}}\in(0,+\infty]$ and a unique maximal solution 
\[
\Gamma_{\mathrm{nl}}\in\mathcal{C}((-T_{\mathrm{min}}^{\mathrm{nl}}, T_{\mathrm{max}}^{\mathrm{nl}}),\gH^s)\,,
\]
 to \eqref{eq:dnl-mb-int}. 
\item\label{pt2.prop-mb}  
For all $m_\sigma, m_\omega, g_\sigma, g_\omega$ there exist $T_{\mathrm{min}}, T_{\mathrm{max}}\in(0,+\infty]$ and a unique maximal solution 
\[
	(\Gamma,S,\omega)\in \mathcal{C}((-T_{\mathrm{min}}, T_{\mathrm{max}}),\gH^s\times H^s(\mathbb{R}^3,\R)\times H^s(\mathbb{R}^3, \R^4) )\,,
\] 
to \eqref{eq:dkg-mb-int}. 
 \end{enumerate}

\end{proposition}

Notice that compared to Proposition \ref{prop:exist}, we did not give a blow-up criterion in the many-body case. This is because this part is new and thus we provide it in a separate statement.

\begin{lemma}\label{lem:blowup-crit-mb}
 Let $s>\tfrac32$, $\Gamma_\mathrm{in}\in \mathfrak{H}^s$ be a non-negative operator, $(S_\mathrm{in},\dot{S}_\mathrm{in})\in H^s(\R^3, \R)\times H^{s-1}(\R^3, \R)$, and $(\omega_\mathrm{in},\dot{\omega}_\mathrm{in})\in H^s(\R^3, \R^4)\times H^{s-1}(\R^3, \R^4)$. Let $\Gamma_{\mathrm{nl}}$ and $(\Gamma,S,\omega)$ the unique maximal solutions to \eqref{eq:dnl-mb-int} and \eqref{eq:dkg-mb-int} given by Proposition \ref{prop:exist-mb}. 
 \begin{enumerate}[label=(\roman*)]
  \item If $T_{\rm max/min}^{\mathrm{nl}}<+\infty$ then for all $s'\in\left(\tfrac32,s\right]$ we have
\[
\limsup_{t\to T_{\rm max/min}^{\mathrm{nl}}}\|\Gamma_{\mathrm{nl}}(t)\|_{\gH^{s'}}= +\infty\,.
\] 
\item If $T_{\rm max/min}<+\infty$ then for all $s'\in\left(\tfrac32,s\right]$ we have
\begin{equation}
	\limsup_{t\to T_{\rm max/min}}\|(\Gamma,S,\omega)(t)\|_{\gH^{s'}\times H^{s'}} 
	= +\infty\,. \label{eq:blowup-crit-mb}
\end{equation}
 \end{enumerate}
\end{lemma}

\begin{proof}
 The proof relies on the application of Lemma \ref{lem:est-density} and Lemma \ref{lem:K-P-mb} in the case $s'\in\left(\tfrac32,s\right]$. To apply these results, one needs to ensure that $\Gamma_{\mathrm{nl}}(t)$ and $\Gamma(t)$ are non-negative operators for all times. This follows from the facts that the initial condition $\Gamma_\mathrm{in}$ is a non-negative operator and that both $\Gamma_{\mathrm{nl}}$ and $\Gamma$ solve a linear von Neumann equation $i\partial_t\Gamma=[D+F(t),\Gamma]$ with a time-dependent potential $F$ (continuous with values in $H^s$), and hence there exists a unitary operator $U(t)$ such that $\Gamma(t)=U(t)^*\Gamma_\mathrm{in} U(t)$ (which shows that non-negativity is propagated along the flow).

 We begin with the blow-up criterion for $\Gamma_{\mathrm{nl}}$. Thus, let $s'\in\left(\tfrac32,s\right]$, $0<T_1<T_{\mathrm{min}}^{\mathrm{nl}}$, $0<T_2<T_{\mathrm{max}}^{\mathrm{nl}}$ and $M>0$ such that 
 $$\|\Gamma_{\mathrm{nl}}\|_{\cC([-T_1,T_2],\gH^{s'})}\leq M.$$
 Defining the non-linear potential 
 $$F(t):=- { \beta}\gamma_\sigma \rho_s(\Gamma(t)) +  \gamma_\omega \rho_v(\Gamma(t)) -  \gamma_\omega {\bm \alpha}\cdot {\bm J}(\Gamma(t)),$$
 we have by Lemma \ref{lem:est-density} and Sobolev embeddings that 
 \begin{align}
  \|F(t)\|_{H^s}&\leq (4\gamma_\sigma + 16 \gamma_\omega) c_{\eqref{eq:est-density}}\|\Gamma(t)\|_{\gH^s}^{1/2}\|\Gamma(t)\|_{\gH^{s'}}^{1/2}
 \leq c_{\eqref{eq:F-mb s-bound}}(M)\|\Gamma(t)\|_{\gH^s}^{1/2} \label{eq:F-mb s-bound}\,\\
 \|F(t)\|_{L^\ii}&\leq C\|F(t)\|_{H^{s'}}\leq C'\|\Gamma(t)\|_{\gH^{s'}}\leq c_{\eqref{eq:F-mb infty-bound}}(M) \label{eq:F-mb infty-bound}.
 \end{align}
 These inequalities together with Lemma \ref{lem:K-P-mb} applied to \eqref{eq:dnl-mb-int} implies that for all $t\in[-T_1,T_2]$,
 \begin{align}
    \|\Gamma_{\mathrm{nl}}(t)\|_{\gH^s} &\leq 
    \|\Gamma_{\mathrm{in}}\|_{\gH^s}+c_{\eqref{eq:K-P-mb}}\int_0^{|t|}\big(\|F(t')\|_{H^s} \|\Gamma(t')\|_{\mathfrak{H}^s}^{1/2} \|\Gamma(t')\|_{\mathfrak{H}^{s'}}^{1/2} + \|F(t')\|_{L^\infty} \|\Gamma(t')\|_{\mathfrak{H}^s} \big)\d t' \notag \\
    &\leq \|\Gamma_{\mathrm{in}}\|_{\gH^s}+c_{\eqref{eq:Gamma_nl blowup}}(M)\int_0^{|t|}\|\Gamma(t')\|_{\mathfrak{H}^s}\,\d t'. \label{eq:Gamma_nl blowup}
 \end{align}
 Gronwall's inequality then shows that $\|\Gamma_{\mathrm{nl}}\|_{\cC([-T_1,T_2],\gH^{s})}\leq \|\Gamma_{\mathrm{in}}\|_{\gH^s} e^{c_{\eqref{eq:Gamma_nl blowup}}(M) \max\{T_1,T_2\}}$, proving that if the solution remains bounded in $\gH^{s'}$, then it also remains bounded in $\gH^s$.
 The argument for the coupled system \eqref{eq:dkg-mb} is similar: if $s'\in\left(\tfrac32,s\right]$, $0<T_1<T_{\mathrm{min}}$, $0<T_2<T_{\mathrm{max}}$ and $M>0$ are such that 
 $$\|(\Gamma,S,\omega)\|_{\cC([-T_1,T_2],\gH^{s'}\times H^{s'})}\leq M,$$
 then defining
 $$F(t):={ \beta}S +  V - {\bm \alpha}\cdot {\bm \omega}$$
 we have 
 $$\|F(t)\|_{H^s}\leq \|(S(t),\omega(t))\|_{H^s},\ \|F(t)\|_{L^\ii}\leq C\|(S(t),\omega(t))\|_{H^{s'}}\leq C M,$$
 and by Lemma \ref{lem:est-density} we also have 
 $$\|(\rho_s(\Gamma(t)),\rho_v(\Gamma(t)),J(\Gamma(t)))\|_{H^s}\leq 20 c_{\eqref{eq:est-density}}\|\Gamma(t)\|_{\gH^{s}}^{1/2}\|\Gamma(t)\|_{\gH^{s'}}^{1/2}\leq 20 c_{\eqref{eq:est-density}} M^{1/2} \|\Gamma(t)\|_{\gH^{s}}^{1/2}.$$
 As a consequence, from \eqref{eq:dkg-mb-int} and Lemma \ref{lem:K-P-mb}, we find that for all $t\in[-T_1,T_2]$,
 \begin{equation}\label{eq:blwp-Gam-dkg}
     \|\Gamma(t)\|_{\gH^s}\leq\|\Gamma_{\mathrm{in}}\|_{\gH^s}+c_{\eqref{eq:blwp-Gam-dkg}}(M)\int_0^{|t|}(\|(S(t'),\omega(t'))\|_{H^s}\|\Gamma(t')\|_{\gH^s}^{1/2}+\|\Gamma(t')\|_{\gH^s})\,\d t',
 \end{equation}
 $$\|(S(t),\omega(t))\|_{H^s}\leq \|(S_{\mathrm{in}},\omega_{\mathrm{in}})\|_{H^s}+\|(\dot{S}_{\mathrm{in}},\dot{\omega}_{\mathrm{in}})\|_{H^{s-1}}+20 c_{\eqref{eq:est-density}} M^{1/2}\int_0^{|t|}\|\Gamma(t')\|_{\gH^{s}}^{1/2}\,\d t'.$$
 Inserting the second bound into the first one, we find using the Cauchy-Schwarz inequality  that, for all $t\in[-T_1,T_2]$, 
 \begin{align}
    \|\Gamma(t)\|_{\gH^s} &\leq \|\Gamma_{\mathrm{in}}\|_{\gH^s}+c_{\eqref{eq:blwp-Gam-dkg-2}}(M)\int_0^{|t|}\Big(\|\Gamma(t')\|_{\gH^s} +\|\Gamma(t')\|_{\gH^s}^{1/2}\notag\\
    &\qquad\qquad\qquad\qquad\qquad+\|\Gamma(t')\|_{\gH^s}^{1/2}\int_0^{|t'|}\|\Gamma(t'')\|_{\gH^s}^{1/2}\,\d t''\Big)\d t'  \label{eq:blwp-Gam-dkg-2}\\
    &\leq \|\Gamma_{\mathrm{in}}\|_{\gH^s}+\tfrac12c_{\eqref{eq:blwp-Gam-dkg-2}}(M)|t|+ c_{\eqref{eq:blwp-Gam-dkg-2}}(M)(\tfrac32+|t|) \int_0^{|t|} \|\Gamma(t')\|_{\gH^s}\d t'. \notag
 \end{align}
 Applying again Gronwall's inequality, we find that $\|\Gamma(t)\|_{\gH^s}$ is uniformly bounded for $t\in[-T_1,T_2]$, which implies the same for $\|(S(t),\omega(t))\|_{H^s}$. This concludes the proof. 
 
 \end{proof}

\begin{remark}
 The reason why we state a blow-up criterion in $H^{s'}$ rather than in $L^\ii$ as in the one-body case is that it is not clear what should be the analog of the space $L^\ii$ for density matrices. 
\end{remark}

\subsection{Uniform estimates}

In order to discuss the limit $m_\sigma, m_\omega \to \infty$ we again consider the equations for 
\begin{equation}\label{eq:def-Sbar-omegabar}
 \overline{S}:=S+\gamma_\sigma \rho_s(\Gamma), \qquad \overline{\omega}:=\omega-\gamma_\omega J(\Gamma),
\end{equation}
with the corresponding initial conditions (c.f.~\eqref{eq:init-red}).
This transforms the equation for $\Gamma$ into
\begin{equation}\label{eq:Gamma-red}
 i\partial_t\Gamma
= \Big[D+ \underbrace{{ \beta}(\overline{S} -\gamma_\sigma \rho_s(\Gamma)) +  \gamma_\omega \rho_v(\Gamma) + \overline{V} -  {\bm \alpha}\cdot(\gamma_\omega  {\bm J}(\Gamma) +\overline{{\bm \omega}}) }_{=:W(\Gamma, \overline{S}, \overline{\omega})}, \Gamma \Big].
\end{equation}

\begin{lemma}\label{lem:red-eq-mb}
 Let $\gamma_\sigma, \gamma_\omega\geq 0$.
\begin{enumerate}[label=(\roman*)]
 \item There are functions $P=(P_\sigma,P_\omega)$ and $Q=(Q_\sigma, Q_\omega)$ such that, for all $m_\sigma, m_\omega>0$,
$(\Gamma,S,\omega)$ is a solution to~\eqref{eq:dkg-mb} with $g_\sigma=\sqrt{\gamma_\sigma}m_\sigma$ and $g_\omega=\sqrt{\gamma_\omega}m_\omega$ if and only $(\Gamma,\overline{S},\overline{\omega})$, where $(\bar{S},\bar{\omega})$ is given by \eqref{eq:def-Sbar-omegabar}, solves the equation
\begin{align}\label{eq:D-F-PQ}
 \left\{\begin{aligned}
        &i\partial_t \Gamma=\big[D + W(\Gamma, \overline{S}, \overline{\omega}),\Gamma] \\
        &(\partial_t^2-\Delta + m_\sigma^2)\overline{S}
	=  P_\sigma\circ (\Gamma, \overline{S}, \overline{\omega}, \nabla\Gamma, \nabla \overline{S}, \nabla \overline{\omega}, \nabla^2 \Gamma) + \partial_t \big(Q_\sigma\circ (\Gamma, \overline{S}, \overline{\omega})\big)\\
        &(\partial_t^2-\Delta + m_\omega^2)\overline{\omega}
	=  P_\omega\circ (\Gamma, \overline{S}, \overline{\omega}, \nabla\Gamma, \nabla \overline{S}, \nabla \overline{\omega},\nabla^2 \Gamma) + \partial_t \big(Q_\omega\circ(\Gamma, \overline{S}, \overline{\omega})\big)\,.
    \end{aligned}
    \right.
\end{align}
 %
%
\item  There exists a constant $c_{\eqref{eq:Q-s-mb-bound}}>0$ such that for all $(\Gamma,\overline{S}, \overline{\omega})$,
  \begin{equation}\label{eq:Q-s-mb-bound}
\| Q\circ(\Gamma,\overline{S}, \overline{\omega})\|_{H^s} \leq c_{\eqref{eq:Q-s-mb-bound}}  \left(\|\Gamma\|^2_{\mathfrak{H}^{s}} + \|\Gamma\|_{\mathfrak{H}^{s}} \|(\overline{S}, \overline{\omega})\|_{H^s}\right)\,.
\end{equation}
\item For all $\tfrac52<s'<s$ and  $M>0$ there exists a constant $c_{\eqref{eq:PQ-mb-bound}}(M)>0$ so that, for all $(\Gamma,\overline{S}, \overline{\omega})\in \mathfrak{H}^s \times H^{s}$ satisfying $\Gamma\geq 0$ and  
\begin{equation*}
 \|(\Gamma,\overline{S}, \overline{\omega})\|_{\mathfrak{H}^{s'}\times H^{s'}} \leq M\,,
\end{equation*} 
the inequality
\begin{equation}\label{eq:PQ-mb-bound}
 \begin{aligned}
 \| Q \circ(\Gamma,\overline{S}, \overline{\omega})\|_{H^s} 
 +&\| P\circ(\Gamma, \overline{S}, \overline{\omega}, \nabla\Gamma, \nabla \overline{S}, \nabla \overline{\omega},\nabla^2 \Gamma)\|_{H^{s-1}} \\
 &\leq c_{\eqref{eq:PQ-mb-bound}}(M)(\|\Gamma\|_{\mathfrak{H}^s}^{1/2}+\|(\overline{S},\overline{\omega} )\|_{H^{s}})\,
\end{aligned}
\end{equation}
holds.
\end{enumerate}
\end{lemma}

\begin{proof}
 Let $F(\Gamma)\in \{\rho_s(\Gamma),- \rho_v(\Gamma), -{\bm J}_k(\Gamma), k=1,2, 3\}$. By the same reasoning as in~\eqref{eq:kg-reduced}, the right hand side of the equations for $\overline{S}, \overline{\omega}$ is given by $\gamma_\bullet (\partial_t^2 - \Delta)F$, for the appropriate choice of $\bullet$ and $F$.
 For a solution $\Gamma(t)$ of~\eqref{eq:dkg-mb} we have, for a self-adjoint $4\times 4$ matrix $A$ determined by the choice of $F$,
 \begin{equation}
  \partial_t F(\Gamma) = -i\tr_{\C^4}\Big(A \big[D +W(\Gamma, \overline{S}, \overline{\omega}) , \Gamma\big](x,x) \Big),
 \end{equation}
that is, one first computes  the commutator, then evaluates the kernel of the resulting operator on the diagonal, and then takes the trace. In analogy with the one-body case, we thus set
\begin{equation}\label{eq:Q-mb-def}
 Q_\bullet:= -i\tr_{\C^4}\Big(A \big[W(\Gamma, \overline{S}, \overline{\omega}) , \Gamma\big](x,x) \Big) =-i\tr_{\C^4}\Big(A \big[W(\Gamma, \overline{S}, \overline{\omega}) , \Gamma(x,x)\big] \Big) .
\end{equation}
Then 
\begin{equation}
 \partial_t^2 F(\Gamma)= -\tr_{\C^4}\Big(A \big[D,[D +W(\Gamma, \overline{S}, \overline{\omega}) , \Gamma]\big](x,x) \Big) + \partial_t Q_\bullet(\Gamma).
\end{equation}
Moreover, since $D_x \Gamma(x,y)= (D\Gamma)(x,y)$, we have
\begin{align}
 D^2 F &= \tr_{\C^4}\Big(A D^2\Gamma(x,x) + 2 A D \Gamma D(x,x) + A \Gamma D^2 (x,x) \Big) \notag \\
&=\tr_{\C^4}\Big(A \big[D,[D , \Gamma]\big](x,x) +4  A D \Gamma D(x,x)\Big).
\end{align}
We thus have
\begin{equation}
 (\partial_t^2 -\Delta)F(\Gamma)= P_\bullet\big(\Gamma, \overline{S}, \overline{\omega}, \nabla\Gamma, \nabla \overline{S}, \nabla \overline{\omega},\nabla^2 \Gamma) +\partial_t Q_\bullet(\Gamma),
\end{equation}
with
\begin{equation}
 P_\bullet=\tr_{\C^4}\Big(-A \big[D,[W(\Gamma, \overline{S}, \overline{\omega}) , \Gamma]\big]+4  A D \Gamma D(x,x) + A\Gamma(x,x) \Big).
\end{equation}
Note that $P_\bullet$ involves second derivatives of $\Gamma(x,y)$, but not second derivatives with respect to the same variable.

The bound~\eqref{eq:Q-s-mb-bound} on $Q_\bullet$ follows from Lemma~\ref{lem:est-density} since $Q$ is linear in $\overline{S}, \overline{\omega}$, quadratic in $\Gamma$, and $H^s$ is an algebra. 

To obtain the bound~\eqref{eq:PQ-mb-bound}, we start from~\eqref{eq:Q-mb-def} and apply the Kato-Ponce inequality~\eqref{eq:K-P}, which gives
\begin{align*}
 \|Q_\bullet\circ(\Gamma, \overline{S}, \overline{\omega})\|_{H^s}
 &\leq 4 \|A\|_{\C^{4\times 4}} \|[W(\Gamma, \overline{S}, \overline{\omega}),\Gamma(x,x)]\|_{H^s} \\
 &\leq 8 \|A\|_{\C^{4\times 4}} c_{\eqref{eq:K-P}}\begin{aligned}[t]
 \Big(\|W(\Gamma, \overline{S}, \overline{\omega})\|_{L^\infty}&  \|\Gamma(x,x)\|_{H^s} \\
 &+ \|W(\Gamma, \overline{S}, \overline{\omega})\|_{H^s}\|\Gamma(x,x)\|_{L^\infty}\Big)
 \end{aligned}
\end{align*}
The Sobolev embedding $H^{s'}\hookrightarrow L^\infty$ and Lemma~\ref{lem:est-density} give, since $W$ is linear in $\Gamma(x,x), \overline{S}, \overline{\omega}$,
\begin{align}
 \|W(\Gamma, \overline{S}, \overline{\omega})\|_{L^\infty} \|\Gamma(x,x)\|_{H^s}
 &\leq c_{\eqref{eq:Q-mb-bound1}}(M)\|\Gamma\|_{\mathfrak{H}^{s}}^{1/2} \label{eq:Q-mb-bound1} \\
 \|W(\Gamma, \overline{S}, \overline{\omega})\|_{H^s}\|\Gamma(x,x)\|_{L^\infty}
 &\leq c_{\eqref{eq:Q-mb-bound2}}(M)(\|\Gamma\|_{\mathfrak{H}^{s}}^{1/2} + \|(\overline{S}, \overline{\omega})\|_{H^s})\label{eq:Q-mb-bound2}.
\end{align}
This is the required bound for $Q$.
For the bound on $P$, first note that Lemma~\ref{lem:est-density} gives
\begin{align}
 \left\|\tr_{C^4}\Big(A D \Gamma D(x,x)\Big)\right\|_{H^{s-1}(\R^3,\C)} 
 &\leq 4 \|A\|_{\C^{4\times 4}}\|D \Gamma D(x,x)\|_{H^{s-1}(\R^3, \C^{4\times 4})} \notag \\
 &\leq  4 c_{\eqref{eq:est-density}} \|A\|_{\C^{4\times 4}}  \|D \Gamma D\|_{\mathfrak{H}^{s'-1}}^{1/2}\|D \Gamma D\|_{\mathfrak{H}^{s}}^{1/2}\notag \\
 &\leq  c_{\eqref{eq:Tr-A bound}}(M)  \|\Gamma \|_{\mathfrak{H}^{s}}^{1/2} \label{eq:Tr-A bound}.
\end{align}
Furthermore, applying the Kato-Ponce inequality~\eqref{eq:K-P} twice, we have 
\begin{align}
  &\left\|\tr_{C^4}\Big(A\big[D,[W(\Gamma, \overline{S}, \overline{\omega}) , \Gamma]\big]\right\|_{H^{s-1}(\R^3,\C)} \notag\\
  &\leq \left\|\tr_{C^4}\Big(A\big[D,W(\Gamma, \overline{S}, \overline{\omega})\big] \Gamma(x,x)\Big)\right\|_{H^{s-1}} + \left\|\tr_{C^4}\Big(A W(\Gamma, \overline{S}, \overline{\omega})\big[D,\Gamma\big] (x,x)\Big)\right\|_{H^{s-1}} \notag\\
  &\leq c_{\eqref{eq:mb-P-KPbound}} \|A\|\Big( \|W(\Gamma, \overline{S}, \overline{\omega})\|_{H^{s}} \|\Gamma(x,x)\|_{W^{1,\infty}} + \|W(\Gamma, \overline{S}, \overline{\omega})\|_{W^{1, \infty}} \|\Gamma(x,x)\|_{H^{s}} \Big). 
  \label{eq:mb-P-KPbound}
\end{align}
for a constant $c_{\eqref{eq:mb-P-KPbound}}>0$. This implies~\eqref{eq:PQ-mb-bound} by bounds analogous to~\eqref{eq:Q-mb-bound1},~\eqref{eq:Q-mb-bound2}.

The equivalence of equations \eqref{eq:dkg-mb} and \eqref{eq:D-F-PQ} follows from these calculations as in Lemma~\ref{lem.red-eq}, and the proof is thus complete.

\end{proof}

In view of Lemma~\ref{lem:red-eq-mb} we have an analogous result to Lemma~\ref{lem:unif-exist} in the many-body case.

\begin{lemma}\label{lem:unif-exist-mb} 
 Let $s>\tfrac52$, $\gamma_\sigma, \gamma_\omega\geq 0$, and $s'\in\left(\tfrac52,s\right]$.
 For $(\Gamma_\mathrm{in}, S_\mathrm{in}, \dot S_\mathrm{in}, \ \omega_\mathrm{in}, \dot \omega_\mathrm{in})$ as in Theorem~\ref{thm.main-mb} set 
 \begin{align*}
  R_0&:= \|\Gamma_\mathrm{in}\|_{\mathfrak{H}^s}
+\|(\overline{\omega}_\mathrm{in}, \overline{S}_\mathrm{in} )\|_{H^s} 
+ \|(\dot{\overline{\omega}}_\mathrm{in}, \dot{\overline{S}}_\mathrm{in})\|_{H^{s-1}} 
+\|Q_\omega(\Gamma_\mathrm{in}, \overline{\omega}_\mathrm{in}, \overline{S}_\mathrm{in} )\|_{H^{s-1}}\\
&\qquad+ \|Q_\sigma(\Gamma_\mathrm{in}, \overline{\omega}_\mathrm{in}, \overline{S}_\mathrm{in} )\|_{H^{s-1}}\,.
 \end{align*}
 For all $T_1,T_2>0$ and for all $M>0$, there exists $M'>R_0$ so that for every solution $(\Gamma, \overline{S},\overline{\omega})\in \mathcal{C}([-T_1, T_2],\mathfrak{H}^s \times H^s)$ to~\eqref{eq:D-F-PQ}
satisfying 
\begin{equation*}
  \| (\Gamma, \overline{S}, \overline{\omega}) \|_{\mathcal{C}([-T_1, T_2],\mathfrak{H}^{s'}\times H^{s'}) } \leq M
\end{equation*}
the inequality
\begin{equation}\label{pt2-unif-mb}
 \|(\Gamma, \overline{S}, \overline{\omega})\|_{\mathcal{C}([-T_1, T_2],\mathfrak{H}^s\times H^s)} \leq M'
\end{equation}
holds.
\end{lemma}

\begin{proof}
We bound $(\Gamma,\bar{S},\bar{\omega})$ using the integral formulation of \eqref{eq:D-F-PQ} which can be inferred from \eqref{eq:D-F-PQ} in the same way that \eqref{eq:dkg-mb-int} is inferred from \eqref{eq:dkg-mb}. In this integral formulation, we use an integration by parts for the terms involving $\partial_t Q$ in the same way that we obtained \eqref{eq:Phidef}.  Let $t\in[-T_1,T_2]$. As in the proof of Lemma \ref{lem:blowup-crit-mb}, recall first that $\Gamma(t)$ is a non-negative operator. We have the bound
 $$
  \|\Gamma(t)\|_{\mathfrak{H}^s} \leq
  \|\Gamma_\mathrm{in}\|_{\mathfrak{H}^s} + \int_0^{|t|} \Big( \|\Gamma W(\Gamma, \overline{S}, \overline{\omega}) \|_{\mathfrak{H}^s} +\|W(\Gamma, \overline{S}, \overline{\omega}) \Gamma  \|_{\mathfrak{H}^s}\Big)(t') \d t'  
 $$
From Lemma~\ref{lem:K-P-mb}, we have 
$$
  \|\Gamma W \|_{\mathfrak{H}^s} +\|W \Gamma  \|_{\mathfrak{H}^s}\leq c_{\eqref{eq:K-P-mb}} \left( \|\Gamma\|_{\mathfrak{H}^s}^{1/2}\|\Gamma\|_{\mathfrak{H}^{s'}}^{1/2} \|W\|_{H^s} + \|\Gamma\|_{\mathfrak{H^s}} \|W\|_{L^\infty} \right).
$$
The form of $W$ (c.f.~\eqref{eq:Gamma-red}), Lemma~\ref{lem:est-density}, and the assumed bound in $H^{s'}$ imply  \begin{align}\label{eq:unif-Gam-mb}
 \|\Gamma\|_{\mathfrak{H}^s}^{1/2}\|\Gamma\|_{\mathfrak{H}^{s'}}^{1/2} \|W\|_{H^s} &\leq  c_{\eqref{eq:unif-Gam-mb}}(M)\Big(\|\Gamma\|_{\mathfrak{H}^s}^{1/2}\|(\overline{S}, \overline{\omega})\|_{H^s} + \|\Gamma\|_{\mathfrak{H}^s}\Big) \\
 \|W\|_{L^\ii}&\leq C\|W\|_{H^{s'}}\leq c_{\eqref{eq:sobobev-inj}}(M). \label{eq:sobobev-inj}
\end{align}
 We thus deduce that 
 $$\|\Gamma(t)\|_{\mathfrak{H}^s} \leq
  \|\Gamma_\mathrm{in}\|_{\mathfrak{H}^s} + c_{\eqref{eq:K-P-mb}}(c_{\eqref{eq:unif-Gam-mb}}(M)+c_{\eqref{eq:sobobev-inj}}(M))\int_0^{|t|} \Big(\|\Gamma\|_{\mathfrak{H}^s}^{1/2}\|(\overline{S}, \overline{\omega})\|_{H^s} + \|\Gamma\|_{\mathfrak{H}^s} \Big)(t') \d t'.  $$
In the same way as in Lemma~\ref{lem:unif-exist}, we deduce from Lemma~\ref{lem:red-eq-mb} that
\begin{align}
 \|(\bar{S},\bar{\omega})(t)\|_{H^s} &\leq 
\|(\overline{\omega}_\mathrm{in}, \overline{S}_\mathrm{in} )\|_{H^s} + \|(\dot{\overline{\omega}}_\mathrm{in}, \dot{\overline{S}}_\mathrm{in})\|_{H^{s-1}} \notag \\
 &\qquad  +\|Q_\omega(\Gamma_\mathrm{in}, \overline{\omega}_\mathrm{in}, \overline{S}_\mathrm{in} )\|_{H^{s-1}} + \|Q_\sigma(\Gamma_\mathrm{in}, \overline{\omega}_\mathrm{in}, \overline{S}_\mathrm{in} )\|_{H^{s-1}} \notag\\
  &\qquad +c_{\eqref{eq:PQ-mb-bound}}(M)\int_0^{|t|}\Big(
 \|\Gamma\|_{\mathfrak{H}^s }^{1/2} +\|(\overline{S},\overline{\omega})\|_{H^s }\Big)(t')\d t'.
 \end{align}
 By Gronwall's inequality, this implies that for all $t\in[-T_1,T_2]$, 
 \begin{align*}
    \|(\bar{S},\bar{\omega})(t)\|_{H^s}\leq \Big(\|(\overline{\omega}_\mathrm{in}, \overline{S}_\mathrm{in} )\|_{H^s} + \|(\dot{\overline{\omega}}_\mathrm{in}, \dot{\overline{S}}_\mathrm{in})\|_{H^{s-1}} +\|Q_\omega(\Gamma_\mathrm{in}, \overline{\omega}_\mathrm{in}, \overline{S}_\mathrm{in} )\|_{H^{s-1}}  \\
    +\|Q_\sigma(\Gamma_\mathrm{in}, \overline{\omega}_\mathrm{in}, \overline{S}_\mathrm{in} )\|_{H^{s-1}}+c_{\eqref{eq:PQ-mb-bound}}(M)\int_0^{|t|}
    \|\Gamma(t')\|_{\mathfrak{H}^s }^{1/2}\,\d t'\Big)e^{c_{\eqref{eq:PQ-mb-bound}}(M)\max\{T_1,T_2\}}.
 \end{align*}
 One can then insert this bound into the bound for $\|\Gamma(t)\|_{\gH^s}$ as in the proof of Lemma \ref{lem:blowup-crit-mb} to deduce the result.
\end{proof}

\subsection{Convergence}

We now turn to the convergence of $\Gamma \to \Gamma_\mathrm{nl}$.

\begin{proposition}\label{prop.conv-mb}
Let $s>\tfrac52$, $0<T_1<T_{\mathrm{min}}^{\mathrm{nl}}$, and $0<T_2<T_{\mathrm{max}}^{\mathrm{nl}}$.  For all $s'\in[s-1,s]$ and $M>0$, there exists $c_{\eqref{eq:Gamma-loc-conv}}(M)>0$ so that 
the following holds. 
For any $m_\sigma, m_\omega\geq 1$ and every solution $(\Gamma, \overline{S}, \overline{\omega})\in\cC([-T_1,T_2],\gH^s\times H^s)$ to~\eqref{eq:D-F-PQ} that satisfies
 \begin{equation*}
  \|(\Gamma, \overline{S}, \overline{\omega})\|_{\mathcal{C}([-T_1,T_2],\mathfrak{H}^s\times H^s)}\leq M,
 \end{equation*}
 we have
\begin{equation}\label{eq:Gamma-loc-conv}
\|\Gamma-\Gamma_{\mathrm{nl}}\|_{\mathcal{C}([-T_1,T_2],\mathfrak{H}^{s'}\times H^{s'})}\leq c_{\eqref{eq:Gamma-loc-conv}}(M)\left(m_\sigma^{s'-s}+m_{\omega}^{s'-s} \right)\,. 
\end{equation}
\end{proposition}
\begin{proof}
 Since the only change to the equations for $\overline{S}$, $\overline{\omega}$ is the replacement of $\Psi$ by $\Gamma$ in the source terms, we have estimates analogous to Lemma~\ref{lem:Phi-error}. We can thus write the fields as $\tilde S$, $\tilde \omega$ plus oscillatory terms, as in~\eqref{eq:tilde-fields}.
 
 We then obtain
\begin{equation}\label{eq:Gamma diff}
 \|\Gamma_{}(t) -\Gamma_{\mathrm{nl}}(t) \|_{H^{s'}}\leq \sum_{k=1}^4\left\|\int_0^te^{-i(t-t')D}I_k (t')e^{i(t-t')D}\d t'\right\|_{H^{s'}}\,,
\end{equation}
with
\[\begin{split}
I_1&= \Big[\left(
	\gamma_\omega 
		 {\bm\alpha}\cdot \big( {\bm J}(\Gamma_\mathrm{nl}) -{\bm J}(\Gamma)\big) 
		 +\gamma_\omega\big(\rho_v(\Gamma) - \rho_v(\Gamma_{\mathrm{nl}})\big)
		-\gamma_\sigma \beta\big(\rho_s(\Gamma) - \rho_s(\Gamma_{\mathrm{nl}})\big)
\right),\Gamma\Big]\,,\\
I_2&= \Big[\left(-\gamma_\omega{\bm\alpha}\cdot {\bm J}(\Gamma_\mathrm{nl})-\gamma_\sigma\beta  \rho_s(\Gamma_{\mathrm{nl}})+\gamma_\omega\rho_v(\Gamma_{\mathrm{nl}})\right),\Gamma-\Gamma_\mathrm{nl} \Big] \,,
\\
I_3&=
\left[\left(-\bm\alpha\cdot \widetilde{\bm \omega}+\beta \widetilde{S}+\widetilde{V}\right),\Gamma\right]\,,
\\
I_4&=\left[\left(\bm\alpha\cdot (\widetilde{\bm \omega}-\overline{{\bm \omega}})+\beta(\overline{S}-\widetilde{S})+\overline{V}-\widetilde{V}\right),\Gamma\right]\,.
\end{split}\]
The terms $I_2, I_2, I_3$ can be estimated in complete analogy with the proof of Proposition~\ref{prop.conv}. The term $I_4$ requires integration by parts. For example, if we consider the contribution due to $\cos\left(t\sqrt{-\Delta + m^2_\sigma}\right)\overline{S}_\mathrm{in}$ (which occurs in $\overline{S}-\tilde S$, see Equation~\eqref{eq:tilde-fields}), this gives
\begin{align*}
 &\int_0^t e^{-i(t-t')D}\Big[\beta \cos\left(t'\sqrt{-\Delta + m^2_\sigma}\right)\overline{S}_\mathrm{in}, \Gamma(t')\Big]e^{i(t-t')D}\d t'  \\
 &= \Big[\beta \frac{\sin(t\sqrt{-\Delta + m^2_\sigma})}{\sqrt{-\Delta + m^2_\sigma}} \overline{S}_\mathrm{in}, \Gamma(t) \Big] \\
 &\qquad - i\int_0^t e^{-i(t-t')D}\Big[D,\Big[\beta \frac{\sin(t\sqrt{-\Delta + m^2_\sigma})}{\sqrt{-\Delta + m^2_\sigma}}\overline{S}_\mathrm{in}, \Gamma(t')\Big]\Big]e^{i(t-t')D}\d t' \\
 &\qquad + i\int_0^t e^{-i(t-t')D}\Big[\beta \frac{\sin(t\sqrt{-\Delta + m^2_\sigma})}{\sqrt{-\Delta + m^2_\sigma}}\overline{S}_\mathrm{in}, \Big[D+W(\Gamma, \overline{S}, \overline{\omega}),\Gamma(t')\Big]\Big]e^{i(t-t')D}\d t'. \\
\end{align*}
By the Jacobi identity, the sum of commutators in the integrand equals
\begin{equation*}
\Big[\Gamma(t'), \Big[D,\beta \frac{\sin(t\sqrt{-\Delta + m^2_\sigma})}{\sqrt{-\Delta + m^2_\sigma}}\overline{S}_\mathrm{in}\Big]\Big]
 + \Big[\beta \frac{\sin(t\sqrt{-\Delta + m^2_\sigma})}{\sqrt{-\Delta + m^2_\sigma}}\overline{S}_\mathrm{in}, \Big[W(\Gamma, \overline{S}, \overline{\omega}),\Gamma(t')\Big]\Big].
\end{equation*}
Its norm in $\mathfrak{H}^{s-1}$, $s-1>\tfrac32$, can thus be bounded by twice
\begin{equation*}
 m_\sigma^{-1}\|S_\mathrm{in}\|_{H^s} \|\Gamma \|_{\mathfrak{H}^{s-1}} + 2m_\sigma^{-1}\|S_\mathrm{in}\|_{H^{s-1}} \|W(\Gamma, \overline{S}, \overline{\omega})\|_{H^{s-1}} \|\Gamma \|_{\mathfrak{H}^{s-1}}.
\end{equation*}
This implies the desired bound on this term using interpolation as in Proposition~\ref{prop.conv}. The calculations for the other terms are analogous, and the proof is then completed by Gronwall's Lemma, as in Proposition~\ref{prop.conv}.

\end{proof}
\subsection*{Proof of Theorem~\ref{thm.main-mb}}
  Let $\tfrac52 <s'<s$, $0<T_1<T_{\mathrm{min}}^{\mathrm{nl}}$, $0<T_2<T_{\mathrm{max}}^{\mathrm{nl}}$. As in the proof of Theorem~\ref{thm.main1}, we first show that there exists $\mu>0$ such that for all $m_\omega,m_\sigma\geq \mu$, we have $T_{\mathrm{min}}>T_1$ and $T_{\mathrm{max}}>T_2$. To do so, let $M>\|(\Gamma_{\mathrm{in}}, \overline{S}_\mathrm{in}, \overline{\omega}_\mathrm{in})\|_{\gH^{s'}\times H^{s'}}$ and define 
  $$T_2' = \sup\{0<t<\min(T_2,T_{\mathrm{max}}),\,\|(\Gamma,\overline{S}, \overline{\omega})\|_{\mathcal{C}([0,t], \gH^{s'}\times H^{s'})}\leq M\}\,.$$
  As in the proof of Theorem~\ref{thm.main1}, we prove that for $M>0$ chosen appropriately, we have $T_2'=T_2$ for $\mu$ large enough, and we do so by contradiction. If $T_2'<T_2$, this again implies that 
  $$\|(\Gamma,\overline{S}, \overline{\omega})\|_{\cC([0,T_2'],\gH^{s'}\times H^{s'})}=M.$$
  Then, by Lemma~\ref{lem:unif-exist-mb}, there exists $M'>0$ such that
\begin{equation}
 \|(\Gamma, \overline{S}, \overline{\omega})\|_{\mathcal{C}([0,T_2'],\mathfrak{H}^s\times H^s)} \leq M'. 
\end{equation}
Using Lemma~\ref{lem:red-eq-mb}, the splitting of $\overline{S}, \overline{\omega}$ into ``small'' and oscillatory parts as in Lemma~\ref{lem:Phi-error} and Gronwall's inequality, we obtain the inequality
\begin{align}
 &\|(\overline{S},\overline{\omega})\|_{\mathcal{C}([0,T_2'],H^{s'})} \notag \\
 &\leq  \Big(\mu^{s'-s}c_{\eqref{eq:S-Gamma-bound}}(M') + \|(\overline{S}_\mathrm{in},\overline{\omega}_\mathrm{in})\|_{H^{s'}} + c_{\eqref{eq:Q-s-mb-bound}} T_2 \|\Gamma\|_{\mathcal{C}([0,T_2'],\mathfrak{H}^{s'})}^2 \Big) e^{c_{\eqref{eq:Q-s-mb-bound}} T_2 \|\Gamma\|_{\mathcal{C}([0,T_2'],\mathfrak{H}^{s'})}}
 \label{eq:S-Gamma-bound}
\end{align}
for some $c_{\eqref{eq:S-Gamma-bound}}(M')>0$.
We then choose 
\begin{equation}
 f(x)=x+ (\|(\overline{S}_\mathrm{in},\overline{\omega}_\mathrm{in})\|_{H^{s'}} + c_{\eqref{eq:Q-s-mb-bound}} T_2x^2)e^{c_{\eqref{eq:Q-s-mb-bound}} T_2 x}
\end{equation}
and 
\begin{equation}
 M=2  f\big(\|\Gamma_\mathrm{nl}\|_{\mathcal{C}([0,T_2],\mathfrak{H}^{s'})}\big).
\end{equation}
Now, by Proposition~\ref{prop.conv-mb} and continuity of $f$,
\begin{align*}
 \| (\Gamma,\overline{S}, \overline{\omega}) \|_{\mathcal{C}([0,T_2'],\mathfrak{H}^{s'}\times H^{s'})}
 &\leq f\big(\|\Gamma\|_{\mathcal{C}([0,T_2'],\mathfrak{H}^{s'})}\big) + \mu^{s'-s}c_{\eqref{eq:S-Gamma-bound}}(M')<M 
\end{align*}
if $\mu$ is large enough. This implies that $T_2'=T_2$ and then Proposition~\ref{prop.conv-mb} gives convergence $\Gamma\to \Gamma_\mathrm{nl}$ on $[0,T_2]$.

\qed

\begin{remark}\label{rk:rate-mb}
As in the one-body case (Remark \ref{rk:rate}), the above proof provides a quantitative rate of convergence. 
\end{remark}

\section*{Acknowledgements}

This work was supported by the Agence Nationale de la Recherche (ANR) through the project DYRAQ ANR-17-CE40-0016 and partly supported by the IRP - CNRS project SPEDO (for L.LT). 
J.L. acknowledges that the ICB is supported by the EUR-EIPHI Graduate School (Grant No. ANR-17-EURE-0002).
%

%

\end{document}